\documentclass[11pt,a4paper]{amsart}

\pdfoutput=1

\usepackage[utf8]{inputenc}
\usepackage[english]{babel}
\usepackage{amsmath,amscd,amssymb}
\usepackage{amsthm}
\usepackage{xcolor}

\usepackage[pdfencoding=auto]{hyperref}  
 \hypersetup{
     colorlinks,
     linkcolor= {blue!80!black},
     citecolor={blue!80!black},
     urlcolor={blue!80!black}}
\usepackage{mathrsfs}
\usepackage{parskip}
\usepackage[left=3cm,right=3cm,top=4.5cm,bottom=4.5cm]{geometry}
\usepackage{enumerate}
\usepackage[nameinlink]{cleveref}
\usepackage{verbatim}

\newtheorem{main-theorem}{Theorem}
\newtheorem{proposition}{Proposition}[section]

\newtheorem{lemma}[proposition]{Lemma}

\theoremstyle{remark}
\newtheorem{remark}[proposition]{Remark}

\numberwithin{equation}{section}

\newcommand{\N}{\mathbb{N}}   
\newcommand{\R}{\mathbb{R}}
\newcommand{\Z}{\mathbb{Z}}                    
\newcommand{\IC}{\mathbb{C}} 
\newcommand{\IS}{\mathbb{S}}
\newcommand{\cV}{\mathcal{V}}
\newcommand{\cR}{\mathcal{R}}

\def\SO{\mathrm{SO}}
\def\so{\mathfrak{so}}
\def\O{\mathrm{O}} 
\def\SS{\mathbb{S}} 
\def\T{\mathsf{T}} 
\def\m{\mathfrak{m}} 
\def\H{\mathfrak{H}} 
\def\A{\mathcal{A}}
\def\cQ{\mathcal{Q}} %

\newcommand{\norm}[1]{\lVert #1 \rVert} 
\newcommand{\br}[1]{\left \langle #1 \right \rangle} 
\newcommand{\ol}[1]{\overline{#1}}

\newcommand{\To}{\longrightarrow}

\newcommand{\gae}{\gamma_{\mathrm{exp}}} 
\newcommand{\qe}{q_{\mathrm{exp}}}

\newcounter{sidenote}
\DeclareMathOperator{\esupp}{ess\,supp}
\DeclareMathOperator{\Spec}{\mathrm{Spec}}
\let\Re\relax
\DeclareMathOperator{\Re}{\mathrm{Re}}
\let\Im\relax
\DeclareMathOperator{\Im}{\mathrm{Im}}
 \DeclareMathOperator{\Id}{Id}
               
\title[The Born approximation in the Calderón problem]{The Born approximation in the three-dimensional Calderón problem}

\author[J. A. Barceló]{Juan A. Barceló} 

\address[JAB]{M$^2$ASAI. Universidad Politécnica de Madrid, ETSI Caminos, c. Profesor Aranguren s/n, 28040, Madrid, Spain.}
\email{juanantonio.barcelo@upm.es}

\author[C. Castro]{Carlos Castro} 
\address[CC]{M$^2$ASAI. Universidad Politécnica de Madrid, ETSI Caminos, c. Profesor Aranguren s/n, 28040, Madrid, Spain.}
\email{carlos.castro@upm.es}

\author[F. Macià]{Fabricio Macià}
\address[FM]{M$^2$ASAI. Universidad Politécnica de Madrid, ETSI Navales, Avda. de la Memoria, 4, 28040, Madrid, Spain.}
\email{fabricio.macia@upm.es}

\author[C. J. Meroño]{Cristóbal J. Meroño}
\address[CJM]{M$^2$ASAI. Universidad Politécnica de Madrid, ETSI Caminos, c. Profesor Aranguren s/n, 28040, Madrid, Spain.}
\email{cj.merono@upm.es}

\begin{document}

\begin{abstract}
Uniqueness and reconstruction in the three-dimensional Calderón inverse conductivity problem can be reduced to the study of the inverse boundary problem for  Schrödinger operators $-\Delta +q $.
We study the Born approximation of $q$ in the ball, which amounts to studying the linearization of the inverse   problem.
 We first analyze this approximation  for real and radial potentials in any dimension $d\ge 3$. 
We show that this approximation is well-defined and obtain a closed formula that involves the spectrum of the Dirichlet-to-Neumann map associated to $-\Delta + q$.  
We then turn to general real and essentially bounded potentials in three dimensions and introduce the notion of averaged Born approximation, which captures the exact invariance properties of the   inverse problem. 
We obtain explicit formulas for the averaged Born approximation in terms of the matrix elements of the Dirichlet to Neumann map in the basis spherical harmonics.
 To show that the averaged Born approximation does not destroy information on the potential, we also  study the high-energy behavior of the matrix elements of the Dirichlet to Neumann map.
\end{abstract}

\maketitle


\section{Introduction and main results}

\subsection{Motivations and setting of the results}
Let $d\ge 2$ and consider a bounded open  set  $\Omega \subset \R^d$ and a function  $\gamma \in  L^\infty (\Omega, \R)$ such that $\gamma(x) \ge \gamma_0>0$ for almost every $x$ in $\Omega$. The solution $u \in H^1(\Omega)$ of the problem
\begin{equation}  \label{id:calderon_direct}
\left\{
\begin{array}{rlr}
-\nabla \cdot  (\gamma \nabla u)=& 0   &\text{ in }   \Omega,\\
u |_{  \partial \Omega} =& g, &    \\
\end{array}\right.
\end{equation} 
models the electrostatic potential generated in the conductor $\Omega$, having conductivity $\gamma$, by a voltage $g$ at the boundary. If $\partial\Omega$ is smooth enough, the exterior normal vector field $\nu$ to $\partial \Omega$ is well-defined and the quantity $ \gamma(x) \partial_{\nu} u(x) = \gamma(x) \nu \cdot \nabla u(x)$ represents the   electric current density flowing through the boundary at $x\in \partial \Omega$. 

The Calderón inverse problem (which dates back to 1980, see \cite{calderon}) consists in determining whether or not one can recover the values of $\gamma$ on $\Omega$  from measurements    of   $\gamma  \partial_{\nu} u $ made at the boundary  $\partial \Omega$ for different choices of $g$ in \eqref{id:calderon_direct}. In order to give a more mathematically precise statement, it is convenient to think of $\gamma  \partial_{\nu} u$ as the value at $g\in C^\infty(\partial\Omega)$ of the linear operator
\[
\Lambda[\gamma]g  := \gamma  \partial_{\nu} u,
\]
where $u$ is the unique solution of \eqref{id:calderon_direct} corresponding to this choice if $g$.
This operator is known as the {\it Dirichlet to Neumann map}, or D-N map for short, since it maps the Dirichlet data to the Neumann data of the boundary value problem \eqref{id:calderon_direct}. It was shown by Sylvester and Uhlmann \cite{SylUhl88} that, at least when $\gamma$ is smooth, that $\Lambda[\gamma] $ is a first order pseudodifferential operator on $\partial\Omega$, and more generally, it always maps $H^{1/2}(\partial \Omega)$ into $H^{-1/2}(\partial \Omega)$.
The Calderón problem can be seen as a model of {\it Electrical Impedance Tomography} (EIT), which  is related to important applications in medical imaging \cite{Assenheimer_2001,cheney99,Isaacson_2006}, or geophysics.

The Calderón problem can be reformulated as determining whether or not the mapping $\gamma\longmapsto\Lambda[\gamma]$ is injective (at least when $\gamma$ varies in some subspace of $L^\infty (\Omega, \R)$) ({\it uniqueness}), and in the affirmative case, in providing appropriate algorithms to recover $\gamma$ explicitly from $\Lambda[\gamma]$ ({\it reconstruction}). The uniqueness and the related question of stability  have driven a considerable amount of research since Calderón's seminal contribution and the early works \cite{KV84, SU87}. The reconstruction aspect has also been the subject of intensive research since the work  \cite{nachman88}. When $d=3$ we mention among many others   \cite{Alessandrini90,Brown,CaroRogers16,GLU2003,Haberman15}, see for example \cite{GKLU2009,uhlmann14} for more references.   We will give more references on the numerical aspects of this problem later on.

A well known simplification, particularly relevant when it comes to designing algorithms for solving the reconstruction problem, is that one can assume that $\Omega$ is a ball in euclidean space without loss of generality. In fact one can also assume that the conductivity is equal to 1 in a neighborhood of the boundary. This is a consequence of the boundary determination results of \cite{Alessandrini90,Brown}, which make possible the construction of an appropriate extension of $\gamma$ to a ball containing $\Omega$. One can recover the D-N map in $\partial \Omega$ from the D-N map in the boundary of the ball.   See \cite[Section 3]{CaroRogers16} for more comments and references on this procedure. 

Another simplification comes from the fact that, if $\gamma\in W^{2,\infty}(\Omega,\R_+)$, then $\nabla \cdot  (\gamma \nabla u)=0$ if and only if $w=\sqrt{\gamma}u$ satisfies $(-\Delta + q)w=0$, where 
\begin{equation}\label{e:qgamma}
q:=\frac{\Delta \sqrt{\gamma}}{\sqrt{\gamma}}.
\end{equation}
From this observation, it is possible to show that the uniqueness question can be reduced to the unique determination of a potential $q\in L^\infty(\Omega,\R)$ from the Dirichlet to Neumann map $\Lambda_q$ associated to the Schrödinger operator $-\Delta+q$ on $\Omega$.
In particular, one has that 
\begin{equation} \label{id:DNmap}
\Lambda_q f = \gamma^{-1/2}  \left( \Lambda[\gamma]  +\frac{1}{2}    \partial_\nu\gamma \right) \gamma^{-1/2} f .
\end{equation}

The approach used since \cite{SU87} to prove to uniqueness and reconstruction of $q$ from $\Lambda_q$ in $d\ge 3$ involves  certain exponentially growing solutions on $\Omega$ of  the equation $-\Delta u + qu =0$ 
known as {\it Complex Geometrical Optic} solutions   or   CGOs.
To describe these solutions, we define the following algebraic variety  of $\IC^d$:
\begin{equation} \label{id:Vd}
\mathcal V(d) = \{ \zeta \in \IC^d : |\zeta|>0,\; \zeta \cdot \zeta = 0\}.
\end{equation}
Notice that $\zeta \in \mathcal V(d)$ if and only if $|\Re(\zeta)| = |\Im(\zeta)|$ and $\Re(\zeta) \cdot \Im(\zeta) =0$; write 
\[
e_\zeta(x)=e^{\zeta\cdot x},\qquad x\in\R^d,
\]
for the exponential harmonic functions introduced in this context by Calderón in \cite{calderon}. Given $h>0$ and $\zeta \in\mathcal V(d) $, a CGO is a family of functions $\psi_\zeta^h \in H^1(\Omega)$ that solve $-\Delta \psi_\zeta^h + q\psi_\zeta^h =0$ on $\Omega$ such that    
\begin{equation*} 
\psi_\zeta^h  = e_{\zeta/h}(1+ r_{h,\zeta,q}), \qquad \lim_{h\to 0^+}\norm{r_{h,\zeta,q} }_{L^2(\Omega)} = 0.
\end{equation*} 
Notice that when $\zeta \in \mathcal V(d)$, the exponential $e_{\zeta/h}$ is a harmonic function in $\R^d$ whose $L^2(\Omega)$-norm diverges as $h\to 0^+$.

The origins of CGOs date back to Faddeev \cite{Fadd65},  who introduced similar objects in the context of Scattering Theory. They have been shown to exist under different regularity assumptions on the potential/conductivity. Their  importance stems from the fact that, as first proved by Sylverster and Uhlmann \cite{SU87}, if one extends $q$ by zero to the whole of $\R^d$ then its Fourier transform satisfies:
\begin{equation}  \label{eq:fourier_from_DNmap}
\widehat{q}(\xi) =    \lim_{h\to 0^+} \br{ e_{{\zeta_1/h}}   ,(\Lambda_q-\Lambda_0) \psi_{\zeta_2}^h}_{H^{1/2}(\partial \Omega)\times H^{-1/2}(\partial \Omega)}, 
\end{equation}
for any $\zeta_1,\zeta_2 \in V(d)$ such that $\zeta_1 + \zeta_2 = -ih\xi $. We remark that \eqref{eq:fourier_from_DNmap} also holds when 
\begin{equation} \label{id:zeta_error}
\zeta_1 + \zeta_2 = -ih\xi + ir_h
\end{equation}
provided that $r_h\in \R^d$ and $|r_h| = o(h)$ as $h\to 0^+$.
  Here we are following the convention
\begin{equation*}  
\widehat{f}(\xi) =  \mathcal F(f) (\xi)=   \int_{\R^d} f(x) e^{-ix\cdot \xi} \, dx,
\end{equation*}
for the Fourier transform of a function $f\in L^1(\R^d)$.

In order to use \eqref{eq:fourier_from_DNmap} to reconstruct  $q$ in $d=3$, one needs first to know the boundary values of $\psi_{\zeta}^h$. This can be obtained  by solving a boundary integral equation  using only the $\Lambda_q$ map, as shown by Nachman in \cite{nachman88}. From the numerical point of view,  a number of challenges appear when trying to implement this scheme, not  least the exponentially growing and oscillatory nature of the CGO solutions as $h\to 0^+$. See for example \cite{Delbary_2011,DHK12,DelbaryKnudsen14} for the case of $d=3$, and \cite{MuellerSiltanen20,MuellerSiltanenBook} for more references on the numerical aspects of EIT. 

A possible strategy to avoid those difficulties is to proceed as it is customary in Scattering Theory, and introduce the so-called Born approximation $\qe$ of the potential: the linearization of the inverse problem. We can formally define $\qe$ by the expression
\begin{equation}  \label{id:q_exp}
\widehat{\qe}(\xi) =  \lim_{h \to 0^+}   { \left  \langle e_{\zeta_1/h}, (\Lambda_q -\Lambda_0)e_{\zeta_2/h} \right  \rangle}    ,
\end{equation}
 where $\zeta_1,\zeta_2 \in \mathcal V(d)$ are taken as in \eqref{id:zeta_error} and 
\begin{equation}  \label{id:dual}
 \left\langle f,g\right\rangle :=\int_{\partial\Omega}f(x)g(x) \, dS(x),\quad \forall f,g\in L^2(\partial\Omega).
\end{equation}
Here the use of the  duality pairing $\left  \langle \ \centerdot \, , \,\centerdot \,  \right \rangle_{H^{1/2}(\partial \Omega)\times H^{-1/2}(\partial \Omega)}$ is not necessary due to the smoothness of $e_\zeta$.

The Born approximation $\qe$  has been considered mainly as a tool to obtain numerical results in the reconstruction problem, see \cite{BKM11,KnudsenMueller11}.  Note that  a notion of Born approximation  can be introduced also in the case of conductivities. In particular, one can obtain a Born approximation $\gae$ for the conductivity problem from $\qe$ by linearizing \eqref{e:qgamma} (see \cite{BKM11,KnudsenMueller11} for this and other approaches). Therefore, obtaining explicit formulas for \eqref{id:q_exp} in terms of the D-N map also yields to interesting consequences for the  inverse conductivity problem. Note also that the linearization of the conductivity problem was already studied by Calderón in \cite{calderon}. The effectiveness of the Born approximation  has been recently compared  to other  numerical reconstruction methods    in \cite{HIKMTB21}  using synthetic data to simulate real discrete data from electrodes. 

As far as we know, even the existence of the  limit $h\to 0^+$ in \eqref{id:q_exp} is a priori not clear. Also, the limit could depend on the choice of  $\zeta_1,\zeta_2 \in \mathcal V(d)$ satisfying \eqref{id:zeta_error}.
In this work we will address these questions and provide explicit formulas to compute $\widehat{\qe}$ from the matrix elements of the D-N map when $\Omega$ is a ball. 
As we mentioned before, working in the ball does not involve any loss of generality when it comes to solve the reconstruction problem.

With these considerations in mind, in the rest of this article we shall assume that $\Omega $ is the ball $B = \{x\in\R^d: |x| \le 1\}$. 
We will denote the boundary of $B$, the unit sphere, by $\SS^{d-1}$ and consider the family 
\begin{equation} \label{id:Qdef}
\cQ_d:=\{q\in L^\infty(B,\R) \,:\, 0\not\in \Spec_{H^1_0(B)}(-\Delta+q)\}. 
\end{equation}
Note that $\cQ_d$ contains all the $q\in L^\infty(B,\R)$ that are obtained from some $\gamma\in W^{2,\infty}(B,\R_+)$ via \eqref{e:qgamma}. 
Our main object of study is the Dirichlet to Neumann map $\Lambda_q$, indexed by potentials $q\in \cQ_d$,  that associates to a function $g\in C^\infty(\SS^{d-1})$ the normal derivative $\Lambda_q g:=\partial_\nu u|_{\SS^{d-1}}$ of the unique solution $u\in H^1(B)$ of
 \begin{equation}  \label{id:calderon_q}
\left\{
\begin{array}{rlr}
-\Delta u + qu  =& 0  & \text{in }   B,\\
u |_{  \SS^{d-1}} =& g. &    \\
\end{array}\right.
\end{equation}
 
The purpose of this work is to investigate the structure of the operators $\Lambda_q$, partly motivated by the reconstruction problem.
In order to do so, we will first show that one can obtain explicit formulas for  the Born approximation  in terms of matrix elements of the D-N map.  From the numerical point of view, this avoids the difficulties of dealing with the exponentially growing and oscillatory  $e_{\zeta/h}$ functions and the limit $h\to 0^+$. We then examine the structure of the matrix elements of the  operator $\Lambda_q$ in the basis of spherical harmonics. In the particular case of radial potentials, they coincide with the eigenvalues of the D-N map; therefore we will examine this case in detail in order to get some insight and motivate our results for the general case.

\subsection{The Calderón problem on the ball. The radial case}

Recall that spherical harmonics are restrictions to the sphere $\SS^{d-1}$ of the complex homogeneous polynomials of $d$ variables that are harmonic. We denote the subspace of spherical harmonics of degree $k$   by $\H_k$.
Spherical harmonics of different degrees are orthogonal in $L^2(\SS^{d})$. Moreover, if  $\varphi_k \in\H_k$ then:
\begin{equation}\label{e:eigenfs}
-\Delta_{\SS^{d-1}} \varphi_k = k(k+d-2) \varphi_k.
\end{equation}
We denote by $\SO(d)$ the special orthogonal group in dimension $d$. For every $Q\in \SO(d)$ write: 
\begin{equation}\label{e:soaction}
R_Q f(x) :=  f(Q^\T(x)),\quad \forall f\in L^2(\SS^{d-1}).
\end{equation}
Then $R^{ -1}_Q = R_Q^*=R_{Q^\T}$ and one can verify that
\begin{equation} \label{id:DN_invariance_int}
R_Q \Lambda_q R_{Q^\T} =    \Lambda_{R_Q (q)} ,\quad \forall q\in\cQ_d.
\end{equation}
In particular, if  $R_Q q=q$ for every $Q\in \SO(d)$ then $[R_Q ,\Lambda_q]=0$ for any rotation. 

This means that $q$ is a radial function and that $\Lambda_q|_{\H_k}=\lambda_k[q] \Id_{\H_k}$ for every $k\in\N_0$, where $\N_0$ stands for the set of non-negative integers (see \Cref{sec:sphere_d}). Note that, since  $\Lambda_q$ is self-adjoint over $L^2(\SS^{d-1})$, the eigenvalues $\lambda_k[q]$ are real and $\Lambda_q=\chi_q(-\Delta_{\SS^{d-1}})$ for some measurable function $\chi_q$ defined on $[0,\infty)$. 

When $q=0$ one has  $\Spec\Lambda_0=\N_0$ and, in view of \eqref{e:eigenfs}, 
\begin{equation}
\Lambda_0=\chi_0(-\Delta_{\SS^{d-1}}),\quad \chi_0(\sigma):=\sqrt{\sigma-\frac{(d-2)^2}{4}}-\frac{d-2}{2}.
\end{equation}
In fact, if $u$ solves \eqref{id:calderon_q} with $q=0$ and $g = \varphi_k$ for any $\varphi_k \in \H_k$, then $u$ must coincide with the solid spherical harmonic $u(x) = |x|^{k} \varphi_k(x/|x|)$. Since
\[
\partial_\nu u(x)=x\cdot\nabla u(x),\quad x\in\SS^{d-1},
\]
one obtains by differentiation:
\begin{equation} \label{id:lambda_0}
\Lambda_0(\varphi_k) = k \varphi_k.
\end{equation}
Note that when $d=2$, $\Lambda_0$ coincides with the fractional Laplacian  $(-\Delta_{\SS^{d-1}})^{1/2}$.

In the following theorem we prove that the limit \eqref{id:q_exp} converges in the radial case
by providing a closed formula. Moreover, we compute explicitly the difference between the Born approximation and the Fourier transform of the potential. Define the moments of the potential as:
\begin{equation} \label{id:moments_radial}
\sigma_{k,1}[q] = \frac{1}{|\SS^{d-1}|} \int_{B} q(x)|x|^{2k} \, dx,
\end{equation}
where ${|\SS^{d-1}|}$ is the volume of the sphere.
\begin{main-theorem}\label{main_thm:Born_aprox_radial} 
Let $d\ge 3$ and $q \in \cQ_d$ be of the form
 $q = q_0(|\cdot|)$, where $q_0 \in L^\infty([0,1],\R)$. 
Let $\xi\in\R^d$ and $\zeta_1^h,\zeta_2^h \in \mathcal V(d)$  such that $\zeta_1^h + \zeta_2^h = -ih\xi + ir_h$ for $h>0 $ small enough, with $r_h\in\R^d$ and $|r_h|=o(h)$ as $h\to 0^+$. 
Then the Born approximation $\widehat{\qe}(\xi)$ given by \eqref{id:q_exp}  is well-defined. \\ Moreover, if $(\lambda_k[q])_{k\in\N_0}$  denote the eigenvalues of the D-N map then:
\begin{equation} \label{id:qexp_formula_int}
\widehat{\qe}(\xi)
 =  2 \pi^{d/2}  \sum_{k=0}^\infty  \frac{ (-1)^k}{k! \Gamma(k+d/2)}
  \left(\frac{|\xi|}{2}\right)^{2k}   (\lambda_k[q] - k) ,
\end{equation} 
and
\begin{equation} \label{id:q_exp_q}
\widehat{\qe}(\xi) -   \widehat{q}(\xi)  = 2\pi^{d/2}  \sum_{k=0}^\infty \frac{(-1)^k}{k!\Gamma(k + d/2)} \left ( \frac{|\xi|}{2}\right )^{2k} (\lambda_k[q] - k - \sigma_{k,1}[q]).
\end{equation}
If, in addition, $r_h=0$ then no limit needs to be taken in \eqref{id:q_exp}:
\begin{equation}  \label{id:q_exp_exact}
\widehat{\qe}(\xi) =   \left  \langle e_{\zeta_1^h/h}, (\Lambda_q -\Lambda_0)e_{\zeta_2^h/h} \right  \rangle.
\end{equation}
\end{main-theorem}

Identity \eqref{id:q_exp_exact} was conjectured by \cite{BKM11} based on numerical evidence.
Note also that an analogue of \eqref{id:qexp_formula_int} has been given in \cite{KLMS2007} and \cite[equation (41)]{IMS2000}  in the two-dimensional case. 

 We remark that the series \eqref{id:qexp_formula_int} converges absolutely, and thus $\widehat{\qe}(\xi)$ is well defined, provided that $\lambda_k[q]-k$ grows as a power of $k$. This is a well known fact, see however our next result \Cref{main_thm:radial_series} below for more precise spectral asymptotics. On the other hand, note that \eqref{id:qexp_formula_int} is a formal definition of $\qe$ in the following sense: even if the series on the right hand side it is absolutely convergent for all $\xi \in \R^d$, there is not an a priori control of the growth in the $\xi$ variable necessary to define the inverse Fourier transform of \eqref{id:qexp_formula_int}.

Identity \eqref{id:q_exp_q} follows immediately from the fact that whenever $q \in L^1(\R^d)$ is compactly supported, 
\begin{equation} \label{id:fourier_from_moments}
 \widehat{q}(\xi)  = 2\pi^{d/2}  \sum_{k=0}^\infty \frac{(-1)^k}{k!\Gamma(k + d/2)} \left ( \frac{|\xi|}{2}\right )^{2k} \sigma_{k,1}[q],
\end{equation}
where the series on the right is absolutely convergent for all $\xi \in \R^d$, see  \Cref{sec:Born_radial} for a short proof of this formula. The fact that $q$ has compact support is essential here: the above formula is false for general Schwartz class functions.  
The similarity between identities  \eqref{id:fourier_from_moments} and \eqref{id:qexp_formula_int} implies that  the Born approximation satisfies {\it formally} that $\sigma_k[\qe] =\lambda_k[q]-k$, and hence it follows   that the linearization of the Calderón problem for the Schrödinger operator $-\Delta +q$ is a Hausdorff moment problem.

This theorem, whose proof is presented in \Cref{sec:Born_radial}, motivates studying the spectrum of the D-N map in the radial case. Our next results gives a perturbation series for the eigenvalues of $\Lambda_q$ for a radial potential $q\in \cQ_d$. 
\begin{main-theorem} \label{main_thm:radial_series} 
Let $d \ge 2$ and $q \in \cQ_d$ be of the form
 $q = q_0(|\cdot|)$, where $q_0 \in L^\infty([0,1],\R)$. 
Let $\alpha:=\max \esupp q_0$.  Then, for all $k > \alpha \norm{q_0}_{L^\infty([0,1])}^{1/2} - \frac{d-2}{2}$ the eigenvalues $\lambda_k[q]$ of $\Lambda_{q}$ can be written as a series
\begin{equation} \label{id:eigenvalue_series}
\lambda_k[q]   =   k +  \sigma_{k,1}[q]+\sum_{n=2}^\infty \sigma_{k,n}[q], 
\end{equation}
where   the real numbers $\sigma_{k,n}[q]$ can be explicitly computed and satisfy that 
\begin{equation} \label{est:radial_sigma_terms}
 |\sigma_{k,n}[q] | \le  \frac{\alpha^{2(n+k) + d-2}}{2 \left(k + \frac{d-2}{2}\right)^{2n-1}}    \norm{q}_{L^\infty(B)}^n, \qquad n\ge 2.
\end{equation}
\end{main-theorem} 

\begin{remark}
The explicit formula to compute $\sigma_{k,n}[q]$ is identity \eqref{id:series_terms} in \Cref{sec:radial_series}. For instance, the case $n=2$ corresponds to:
\begin{equation} \label{id:momentsn2_radial}
\sigma_{k,2}[q] = \frac{1}{k + \frac{d-2}{2}}\left(\frac{1}{2} \sigma_{k,1}[q]^2 -\int_{0}^\alpha\int_{0}^r q_0(r)q_0(s)s^{2k+d-1}r \, ds\,dr \right),
\end{equation}
\end{remark}
\begin{remark}\label{rem:eigenvest}
Summing the series \eqref{id:eigenvalue_series} and applying the bound \eqref{est:radial_sigma_terms} gives the estimate
\begin{equation} \label{est:moment_aprox_radial}
\left| \lambda_k[q] -k  - \sigma_{k,1}[q]  \right |\leq \alpha^{d+2} \|q\|^2_{L^\infty(B)}\frac{\alpha^{2k}}{2(k + (d-2)/2)^3} , 
\end{equation}
which is of lower order in $k$ than:
\[
|\sigma_{k,1}[q]|\leq \alpha^{d} \|q\|_{L^\infty(B)} \frac{\alpha^{2k}}{2k+d}.
\]
\end{remark}
\begin{remark}
It is well known (see \cite{SylUhl88}) that whenever $q$ is smooth and compactly supported in the interior of $B$, $\Lambda_q$ is a pseudodifferential operator that is smoothing to all orders. \Cref{main_thm:radial_series} quantifies the structure of the exponentially small remainder in the radial case. The fact that the exponential rate depends on the distance of the essential support to the boundary of $B$ was already noticed in \cite{Mand00}.
\end{remark}

Let us note that the spectral theory of D-N operators is a very active subject, see for instance the survey article \cite{GirPol17} and the references therein. See also \cite{daude1,daude2} for recent results in a context closer to that of \Cref{main_thm:radial_series}.

\Cref{main_thm:Born_aprox_radial} shows in particular that from the knowledge of the moments $\sigma_{k,1}[q]$ for all $k\in \N_0$ one can recover completely the Fourier transform of $q$. On the the other hand, \eqref{est:moment_aprox_radial} shows that  the moments  $\sigma_{k,1}[q]$ can be approximated by $\lambda_k[q] -k$ modulo an error of order $k^{-2}\sigma_{k,1}[q]$.
Heuristically  this gives evidence that 
the Born approximation captures the high frequency information of the potential, as it is well known in scattering problems (see \Cref{sec:Born_radial} for more  references). In a similar spirit, the recovery of jumps of discontinuous conductivities from the linearization of the inverse problem has been  recently studied  in the two dimensional case in \cite{GLSSU2018} . 




\subsection{The Calderón problem on the ball. General Schrödinger operators}

We now address the more complex non-radial case when $d=3$. If one tries to follow the strategy we implemented in the radial case, the resulting expression in  \eqref{id:q_exp} involves all matrix elements of $\Lambda_q$, some of them with divergent coefficients as $h \to 0^+$. To avoid this problem we perform an averaging procedure that leads to a convergent expression. This average also has the advantage of reducing the number of matrix elements involved in the final result, yielding  a formula that could be used in practical applications. The averaging procedure is motivated by the following observation. Given $\omega\in\SS^2$ the Fourier transform of $q$ which has the following invariance property
\begin{equation} \label{id:invariance_fourier}
 \widehat{R_Q(q)}(s \omega)  = \widehat{q}(s \omega)  \quad \text{ for all } s>0 \text{ and } Q\in \SO(3)_\omega,  
\end{equation}
where, if $\omega \in \SS^{d-1}$,  the set $\SO(d)_\omega$ is the stabilizer of $\omega$:
\begin{equation} \label{id:stabilizer}
\SO(d)_\omega=\{Q\in\SO(d)\,:\, Q(\omega)=\omega\}.
\end{equation}

On the other hand, one expects  the Born approximation to enjoy the same symmetry properties that the Fourier transform has. Recalling \eqref{id:q_exp} however,  
for every $Q\in \SO(3)_\omega$ and $\zeta_1,\zeta_2$ as in \eqref{id:zeta_error}, one has that
 \begin{equation}    \label{id:rotat_invariace_B}
        \left  \langle e_{ \zeta_1/h}, (\Lambda_{R_Q(q)} -\Lambda_0)e_{  \zeta_2/h} \right  \rangle = \left  \langle e_{ Q^\T(\zeta_1/h)}, (\Lambda_q -\Lambda_0)e_{ Q^\T(\zeta_2/h)} \right  \rangle,
\end{equation}
where the pair $Q^\T(\zeta_1)$, $Q^\T(\zeta_2)$ also satisfies \eqref{id:zeta_error} in place of $\zeta_1,\zeta_2$ since $Q \in \SO(3)_\omega$. This implies that, prior to taking the limit $h\to 0^+$, the Born approximation (or more precisely, the scattering transform) does not necessarily share the invariance property \eqref{id:invariance_fourier}.  In fact, in general,  the right hand side of \eqref{id:rotat_invariace_B}  differs from
\[ 
\left  \langle e_{ \zeta_1/h}, (\Lambda_{q} -\Lambda_0)e_{  \zeta_2/h} \right  \rangle.
\]
This motivates the introduction of an averaged version of the Born approximation. In order to do so, let:
\[
L_{3} := \frac{1}{i} (x_1\partial_{x_2}-x_2\partial_{x_1}),
\]
and, for $\omega\in\SS^2$ define the \textit{angular momentum with respect to} $\omega$ as:
\[
L_\omega = R_{P^\T} L_3 R_P,\quad \text{ where }P\in\SO(3)\text{ satisfies }P\omega = (0,0,1).
\]
This is a self-adjoint operator on $L^2(\SS^2)$ whose spectrum is equal to $\Z$ and whose definition is independent of the particular choice of $P$. Therefore, it generates a unitary group $e^{-itL_\omega}$, $t\in \R$, that is $2\pi\Z$-periodic; in fact it generates the stabilizer of $\omega$:
\begin{equation}\label{e:stabilizer}
    \SO(3)_\omega=\{e^{-itL_\omega} \,:\, t\in [0,2\pi]\},
\end{equation}
see \Cref{sec:spherical_harmonics_d3} for more details. We define the averaged operator:
\begin{equation} \label{id:averaged_operator}
\langle\Lambda_q\rangle_\omega:= \frac{1}{2\pi}\int_0^{2\pi} e^{itL_\omega} \Lambda_q e^{-itL_\omega} dt.
\end{equation}
Note that, by construction, $[\langle\Lambda_q\rangle_\omega,L_\omega]=0$. The operator $\langle\Lambda_q\rangle_\omega$ is known as the quantum average of $\Lambda_q$ along the flow $e^{-itL_\omega}$. This construction is the starting point for implementing the averaging method in quantum mechanics that goes back to \cite{Wein77} and has been used extensively in the spectral theory, see among many others, \cite{Guill78, Ur84, Ur85, OjVil12, GuillUrWang12, MaRiv16}.

We are now in position to define formally the \textit{averaged Born approximation}  $\langle  \qe \rangle $   by the formula
\begin{equation}   \label{id:qexpA}
 \widehat{ \langle \qe \rangle} (\xi )   
=  \lim_{h\to 0^+} \langle e_{\zeta_1/h},   (\langle\Lambda_q\rangle_{\xi/|\xi|} -\Lambda_0) e_{\zeta_2/h}\rangle,
\end{equation}
for any $\zeta_1,\zeta_2 \in \cV(3)$ such that $\zeta_1 + \zeta_2 = -ih\xi + ir_h$, with $r_h\in \R^d$ and $|r_h| = o(h)$ as $h\to 0^+$. As in the case of the Born approximation, the existence of this limit (and the  possible dependence on the choice of $\zeta_1,\zeta_2$) is not clear \textit{a priori}. The motivation behind this definition follows from the next remark.
\begin{remark}\label{rem:dtnav}
Note that, because of \eqref{id:DN_invariance_int} and \eqref{e:stabilizer},
\begin{equation} \label{id:first_identity}
\langle\Lambda_q\rangle_{\omega} = \frac{1}{2\pi}\int_0^{2\pi}\Lambda_{e^{itL_\omega}(q)}\,dt  = \int_{\SO(3)_\omega} \Lambda_{R_Q(q)}  \, d\mu(Q),
\end{equation}
where $\mu$ is the normalized Haar measure of $\SO(3)_\omega\simeq \SO(2)$. Therefore, using \eqref{id:rotat_invariace_B}, one concludes that:
\[
 \langle e_{\zeta_1/h},   (\langle\Lambda_q\rangle_{\omega} -\Lambda_0) e_{\zeta_2/h}\rangle= \int_{\SO(3)_\omega}\left  \langle e_{ Q^\T(\zeta_1/h)}, (\Lambda_q -\Lambda_0)e_{ Q^\T(\zeta_2/h)} \right  \rangle d\mu(Q),
\]
 where $\omega = \xi/|\xi|$. This justifies the terminology of averaged Born approximation:  we are just inserting an average before taking  the limit in \eqref{id:q_exp} over all the rotations of $\zeta_1$ and $\zeta_2$ that preserve \eqref{id:zeta_error}.
\end{remark}
The previous identities also show that  for all  $Q\in \SO(3)_\omega$ and $s>0$
 \begin{equation*}  
 \widehat{\langle h_{\mathrm{exp}} \rangle }(s \omega)  = \widehat{\langle \qe \rangle }(s \omega)  \quad \text{if } h= R_Q(q).
\end{equation*}
This is the analogue of \eqref{id:invariance_fourier}.
Moreover, \Cref{rem:dtnav} implies that, as soon as $q$ is invariant by $\SO(3)_\omega$
\[
   \widehat{ \langle \qe \rangle} (s\omega )   = \widehat{\qe}(s\omega  ) ,\qquad \text{for all } s>0.
\]

As we will later justify, this averaging procedure does not destroy any relevant information on the potential $q$. 
In order to state the main result of this section, we specify an orthonormal basis of spherical harmonics, namely the one that diagonalizes the angular momentum operators $L_\omega$. Since $L_\omega$ commutes with $\Delta_{\SS^{2}}$, it turns out that the self-adjoint operator $L_\omega$ leaves $\mathfrak{H_\ell}$ invariant. It can be proved that the spectrum of $L_\omega|_{\mathfrak{H_\ell}}$ is simple and is given by the integers $m$ such that $-\ell\leq m \leq \ell$. We will denote $Y_{\ell,m}^\omega$, with $-\ell\leq m \leq \ell$, the corresponding basis of eigenfunctions:
\[
L_\omega Y_{\ell,m}^\omega = m Y_{\ell,m}^\omega,\qquad Y_{\ell,m}^\omega\in\H_\ell.
\]
More details on the angular momentum and the specific choice of the $Y_{\ell,m}^\omega$ eigenfunctions are provided in   \Cref{sec:spherical_harmonics_d3}.

 We start by defining the following matrix elements of the D-N map
\begin{equation}   \label{id:matrix_elements} 
 \lambda_{k,\ell;\omega}[q] = \int_{\IS^{2}} \ol{Y_{\ell,k}^\omega}(x) \Lambda_q Y_{k,k}^\omega(x) \, dS(x)  \quad \text{for }  k,\ell \in \N_0,\; \ell \ge k,   
\end{equation}
for any $\omega \in \SS^2$, and we now define the following moments of $q$
\begin{equation} \label{id:moments}
\m_{k,\ell;\omega}[q] = \int_{B} |x|^{\ell + k} q(x)   \ol{Y_{\ell,k}^\omega}\left(\frac{x}{|x|}\right)  Y_{k,k}^\omega \left(\frac{x}{|x|}\right)\, dx  .
\end{equation}
Notice that  $\m_{k,\ell;\omega}[q] $ is a matrix element in the solid harmonics  of the multiplication operator associated to the potential $q$.

Take $\xi \in \R^3/\{0\}$. Let $\{\eta_1,\eta_2,\omega\}$ be a positively oriented orthonormal basis of $\R^3$, where  $\omega = \xi/|\xi|$. We fix $\zeta_1,\zeta_2 \in \mathcal V(3)$ by
\begin{equation} \label{eq:zeta_h_nonradial}
\begin{alignedat}{1}
  \zeta_1 &=  -\eta_1 - i\left( 1 -    h^2|\xi|^2 \right )^{1/2}\eta_2   -   ih |\xi|\omega,  \\
 \zeta_2 &=  \eta_1 + i\eta_2  .
 \end{alignedat}
\end{equation}
Notice that $\zeta_1 + \zeta_2 = -ih\xi + ir_h$, with   $|r_h| = o(h)$ as $h\to 0^+$, so \eqref{id:zeta_error} is satisfied. 
For the sake of simplicity we also introduce the constants
\begin{equation} \label{id:mu}
\mu_{k,\ell} =  \frac{4\pi  }{\sqrt{(2k+1)(2\ell+1)}} \frac{  1}{ \sqrt{ (2k)!(\ell-k)! (k+\ell)!}},  \qquad \ell \ge k.
\end{equation}   

Our last main result is a closed formula for the averaged Born approximation.
\begin{main-theorem}  \label{main_thm:qexpA}
Consider a potential $q \in \cQ_3$. Take $\xi \in \R^3\setminus\{0\}$ and let $\omega := \xi/|\xi|$. Assume $   \widehat{ \langle \qe \rangle}(\xi)$ is given by \eqref{id:qexpA} with $\zeta_1$ and $\zeta_2$ as in \eqref{eq:zeta_h_nonradial}. 
If $\lambda_{k,\ell;\omega}[q]$ are given by \eqref{id:matrix_elements} then
\begin{equation}  \label{id:qexpA_int}
  \widehat{ \langle \qe \rangle} (\xi) =   
  \sum_{k=0}^\infty   \sum_{\ell=k}^\infty   (-i)^{\ell+k}  \mu_{k,\ell}  \, |\xi|^{\ell+k}  (\lambda_{k,\ell;\omega}[q] - k \delta_{k,\ell})    ,
\end{equation}
where $\delta_{k,\ell}$ denotes the Kronecker delta.  
Also, the $\m_{k,\ell;\omega}[q]$ coefficients defined in \eqref{id:moments} are the linearization of the matrix elements $\lambda_{k,\ell;\omega}[q] - k \delta_{k,\ell}$, and   one has that \footnote{Throughout the paper we write $a \lesssim b$  when $a$ and $b$ are positive constants and there exists $C > 0$ so that $a \leq C b$. We refer to $C$ as the implicit constant in the estimate.}
\begin{multline}  \label{est:asymp_k_int}
\left| \lambda_{k,\ell;\omega}[q] - k \delta_{k,\ell} - \m_{k,\ell;\omega}[q] \right|   \\     
\lesssim   \frac{\alpha^{k+\ell}}{(\ell +1)\sqrt{k+1}}   \norm{ q}_{L^\infty (\R^3)}^2 \left (1 +  \norm{ q}_{L^\infty (\R^3)}\norm{\mathcal (-\Delta  + q)^{-1}}_{L^2(B) \to L^2(B)} \right ) ,
\end{multline}
where  $\alpha:=\max_{x\in \esupp q} |x| $.  
Moreover, for all functions $q\in L^1(\R^3)$ with compact support the following holds
\begin{equation} \label{id:Fourier_representation_int}
\widehat{q} (\xi)   =   
     \sum_{k=0}^\infty  \sum_{\ell=k}^\infty   (-i)^{\ell+k}  \mu_{k,\ell}  \, |\xi|^{\ell+k} \, \m_{k,\ell;\omega}[q] .
\end{equation}
\end{main-theorem}

This theorem shows that the limit \eqref{id:qexpA} is well defined, and \eqref{est:asymp_k_int} and \eqref{id:Fourier_representation_int} show that the average used to define $   \widehat{\langle \qe \rangle}(\xi)$  does not  destroy any relevant information on the potential. 
  As expected, in the radial case one can recover \eqref{id:qexp_formula_int} from formula \eqref{id:qexpA_int}, since $ 2^{2k} k! \Gamma\left(k + 1/2 \right) = \sqrt{\pi} (2k)!$. 
Also, notice that in the case $k=\ell$ the moment  
$\m_{k,k;\omega}[q] $ behaves sharply as $\alpha^{2k}(2k)^{-1}$   |to see this just assume $q$ is the indicator function of the ball $B_\alpha$ with $\alpha\le 1$| so estimate \eqref{est:asymp_k_int}  implies an extra decay of $k^{-1/2}$ in the right hand side.

  Again, we remark that the series \eqref{id:Fourier_representation_int} is absolutely convergent provided $q$ has compact support and this formula fails for general functions in the Schwartz class.
 The right hand side of \eqref{id:qexpA_int} is also an absolutely convergent series, but there is no control on the growth of $ \widehat{ \langle \qe \rangle} (\xi) $ as $|\xi| \to \infty$ to guarantee that it is   a tempered distribution. Therefore we need to consider \eqref{id:qexpA_int} just as a formal definition of $  \langle \qe \rangle   $.


 Let us note that the results in Theorems \ref{main_thm:Born_aprox_radial} -- \ref{main_thm:qexpA} 
 can be used as the basis of  efficient numerical algorithms to reconstruct the Born approximations $\gae$ and $\qe$. This can be used to improve and simplify   the algorithms to reconstruct the potential and the conductivity from the D-N map.  
This is the subject of the article \cite{numerical}.

\subsection{Structure of the article} \Cref{sec:material} presents some well-known facts on spherical harmonics, the action of the group of rotations on the sphere and angular momentum operators that will be used throughout the paper. 
In \Cref{sec:Born_radial} we prove \Cref{main_thm:Born_aprox_radial}, whereas \Cref{sec:main_Born_section_gen} is devoted to the proof of identity \eqref{id:qexpA_int}. 
In \Cref{sec:spectrum} we study the spectrum and matrix elements of the D-N map. 
There we prove \Cref{main_thm:radial_series} and formula \eqref{est:asymp_k_int}. 
Finally, \Cref{sec:fourierTF} presents the proof of the identity \eqref{id:Fourier_representation_int} for the Fourier transform and completes the proof of \Cref{main_thm:qexpA}.

\subsection{Acknowledgments}
This research has been supported by Grant MTM2017-85934-C3-3-P of Agencia Estatal de Investigación (Spain).

\section{A few preliminary facts on spherical harmonics} \label{sec:material}

In this section we summarize some properties of the angular momentum operators and of the spherical harmonics which will be useful later on. All the properties stated here without proofs are well known (see for example  \cite{atkinson,steinweiss,taylor_NCHA}).

\subsection{Spherical harmonics} \label{sec:spherical_harmonics}

The eigenfunctions of  $\Delta_{\SS^{d-1}}$ acting on $L^2(\SS^{d-1})$ are  the restriction to the sphere of the complex homogeneous polynomials of $d$ variables that are harmonic. 
These functions are called the spherical harmonics; the set of spherical harmonics of degree $k$ is denoted by $\H_k$.
Spherical harmonics of different degrees are orthogonal in $L^2(\SS^{d-1})$. If $\varphi\in\H_k$ then:
\begin{equation*}
-\Delta_{\SS^{d-1}} \varphi = k(k+d-2) \varphi.
\end{equation*}
The multiplicity of the eigenvalue $k(k+d-2)$ is equal to:
\[
\dim \H_k=  \binom{d-1+k}{k} -\binom{d+k-3}{k-2},
\]
which if $d=2$ gives $\dim \H_k = 2$, and if $d=3$ gives $\dim \H_k = 2k+1$. One also has that:
\[
\H=\bigoplus_{k=0}^\infty \H_k.
\]
is dense in $L^2(\SS^{d-1})$, and that $\H_k$ is spanned by the restriction to the sphere of the polynomials:
\[
(\zeta \cdot x)^k,\quad \text{ with } \zeta \in \mathcal V(d) ,  
\]
where $\mathcal V(d)$ was defined in  \eqref{id:Vd}. This particular class of spherical harmonics will play an important role in this work, as  we will later see in \Cref{sec:Born_radial,sec:main_Born_section_gen}.
\begin{lemma} \label{lemma:prod_spherial_harmonics}
 Let $d\ge 3$ and $\zeta_1 ,\zeta_2 \in \mathcal V(d) $. Then, for all $k \in \N_0$ we have that
\begin{equation} \label{id:zeta_int}
 \int_{\SS^{d-1}} (\zeta_1 \cdot x)^{k} (\zeta_2 \cdot x)^{k} \, dS(x)  
 =   c_k  \left(\frac{\zeta_1 \cdot \zeta_2}{2} \right )^{k}    ,
\end{equation}
where
  \begin{equation} \label{id:c_k}
 c_k = 2\pi^{d/2}  \frac{k!}{ \Gamma(k+d/2)}  .
\end{equation}
\end{lemma}
\begin{proof}
Let $k\in\N_0$ and define
\[I_k :=   \int_{\SS^{d-1}} (\zeta_1 \cdot \theta)^{k} (\zeta_2 \cdot \theta)^{k} \, dS(\theta) .\]
Then, by \eqref{id:lambda_0} and Green formula we have that
\[
 kI_k  =  \int_{\SS^{d-1}} \Lambda_0\left( (\zeta_1 \cdot \centerdot )^{k} \right )(\zeta_2 \cdot \theta)^{k} \, dS(\theta)
 = \int_{B} \nabla(\zeta_1 \cdot x)^{k} \nabla(\zeta_2 \cdot x)^{k} \, dx.
\]
 Therefore one gets, for $k>0$,
\begin{multline*} 
I_k  
 = k (\zeta_1\cdot \zeta_2 )  \int_{B} (\zeta_1 \cdot x)^{k-1} (\zeta_2 \cdot x)^{k-1} \, dx  \\
 =  k (\zeta_1\cdot \zeta_2 )   \left[ \int_0^1 r^{2k-2} r^{d-1} \, dr \right]  \int_{\SS^{d-1}} (\zeta_1 \cdot \theta)^{k-1} (\zeta_2 \cdot \theta)^{k-1} \, dS(\theta) ,
 \end{multline*} 
 which implies that
 \[
 I_k  =  
 (\zeta_1\cdot \zeta_2 )\frac{k}{2k+d-2} I_{k-1}, \quad \text{for all } k=1,2,\dots 
 \]
Since $ I_0 =  |\SS^{d-1} | $, this shows that the identity \eqref{id:zeta_int} holds with
\[
 c_k = |\SS^{d-1}|   {k!}  \prod_{j=1}^{k}\frac{1}{(j+\frac{d-2}{2})  }.
\]
This finishes the proof since 
\[ 
|\SS^{d-1}|   = \frac{2\pi^{d/2}}{\Gamma(d/2)}, \quad \text{and} \quad  \prod_{j=1}^{k}\frac{1}{(j+\frac{d-2}{2})  } = \frac{\Gamma(d/2)}{\Gamma(k+d/2)},  
\]
which justify \eqref{id:c_k}
\end{proof}
Let $\zeta \in \mathcal V(d)$. If we now  define the spherical harmonics 
\begin{equation} \label{id:Yk}
  Y^k_\zeta(x) = (\zeta \cdot x)^k \qquad x\in \SS^{d-1},
\end{equation}
as a consequence of the previous lemma we have that $\norm{Y^k_\zeta}_{L^2(\SS^{d-1})}^2  = c_k 2^{-k}|\zeta|^{2k} $.
 In the particular case $|\zeta| = \sqrt{2}$ |for example if $\zeta = e_1+ie_2$ with $e_1$ and $e_2$  orthonormal| it follows that 
 \begin{equation} \label{id:Yk_normalization}
 \norm{Y^k_\zeta}_{L^2(\SS^{d-1})}^2  = c_k  .
 \end{equation}

\subsection{Representation of \texorpdfstring{$\SO(d)$}{SO(d)} on \texorpdfstring{$L^2(\SS^{d-1})$}{L2} and angular momentum} \label{sec:sphere_d}

We now briefly recall some properties of the representation of $\SO(d)$ on $L^2(\SS^{d-1})$ introduced in \eqref{e:soaction} (see for example \cite{taylor_NCHA}). Recall that $\so(d)$, the Lie algebra of $\SO(d)$, can be identified to the set of skew-symmetric $d\times d$ matrices. Given $\omega\in\so(d)$, the family $(e^{t\omega})_{t\in\R}$ defines a one parameter subgroup on $\SO(d)$ and $\SO(d)=\{e^\omega\,:\, \omega\in\so(d)\}$. If $f\in C^\infty(\SS^{d-1})$ then:
\[
i\partial_t R_{e^{t\omega}} f =L_{\omega}R_{e^{t\omega}} f,
\]
for some self-adjoint first-order differential operator $L_\omega$ on $L^2(\SS^{d-1})$ that we will call the \textit{angular momentum} associated to $\omega$. Then
\begin{equation}\label{id:exponential_map}
    e^{-itL_\omega}f = R_{e^{t\omega}} f,\qquad \forall f\in L^2(\SS^{d-1}).
\end{equation}
Let $\mathcal{A}$ denote an operator on $L^2(\SS^{d-1})$ such that the space $C^\infty(\SS^{d-1})$ is dense in $D(\A)$ and mapped into itself by $\A$.
Then for every $\omega \in \so(d)$ one has that
\begin{equation*}
    [\mathcal{A},e^{-itL_\omega}]=0, \; \forall t \in \R \quad \text{ if and only if } \quad  [\mathcal{A},L_\omega]=0.
\end{equation*}
This is the case for $\mathcal{A}=\Delta_{\SS^{d-1}}$, and hence it follows that 
\begin{equation*}
e^{-itL_\omega},\, L_\omega :  \H_k \To \H_k,\quad \forall k\in\N_0,\;\forall \omega\in\so(d).   
\end{equation*}
In addition,  using the fact that $\Delta_{\SS^{d-1}}$ can be expressed in terms of the angular momentum operators,  one can prove that an operator $\A$ as above  satisfies
\[
[\A,\Delta_{\SS^{d-1}}]=0,
\]
provided that $[\A,R_Q]=0$ for every $Q\in\SO(d)$. Since  the representation $ Q\mapsto R_Q$ is irreducible when restricted to $\H_k$, Shur's lemma implies that every operator $\mathcal{A}$ that maps $\H_k$ onto itself and commutes with rotations satisfies
\[
\mathcal{A}|_{\H_k}=\lambda_k \Id_{\H_k},\quad \text{ for some } \lambda_k\in\IC.
\]


 \subsection{Spherical harmonics and angular momentum when \texorpdfstring{$d=3$}{d=3} } \label{sec:spherical_harmonics_d3}
 
When $d=3$ the Lie algebra $\so(3)$ is isomorphic to $\R^3$ equipped with the Lie bracket given by the cross product. Because of this, the expression of the angular momentum operators is particularly simple. Any angular momentum operator  is a linear combination of the operators corresponding to the vectors $\{e_1,e_2,e_3\}$ of the canonical basis of $\R^3$:
\begin{equation} \label{id:angular_momentum_cartesian}
\begin{aligned}
L_{1} :=L_{e_1}&= \frac{1}{i} (x_2\partial_{x_3}-x_3\partial_{x_2}) \\
L_{2} :=L_{e_2}&= \frac{1}{i} (x_3\partial_{x_1}-x_1\partial_{x_3}) \\
L_{3} :=L_{e_3}&= \frac{1}{i} (x_1\partial_{x_2}-x_2\partial_{x_1}).
\end{aligned}
\end{equation}
These operators act both on $L^2(\SS^2)$ and $L^2(\R^3)$. 
For simplicity we will denote the operators by the same symbol independently of the space in which they are acting. If $\omega\in\R^3$ then $
    L_\omega = \omega_1 L_1+\omega_2 L_2+\omega_3 L_3$, where  $\omega_1,\dots,\omega_3$, are the components of $\omega$
with respect to the canonical base.
In addition, when $|\omega|=1$ then for every $P\in\SO(3)$ such that $P(\omega) = e_3$ one has:
\begin{equation}\label{e:uequiam}
    L_\omega = R_{P}^* L_3 R_P =  R_{P^\T} L_3 R_P.
\end{equation}
Therefore, all angular momentum operators associated to unitary vectors are unitarily equivalent. 

Consider now the parametrization of $\SS^{2}$ given by the spherical coordinates with north pole $e_3$:
\begin{equation} \label{id:spherical_coord}
x(\theta,\phi) = (\sin \theta \cos\phi,\sin \theta \sin \phi,\cos\theta), \qquad (\theta,\phi) \in (0,\pi)\times(0,2\pi).
\end{equation}  
In these  coordinates the angular momentum operator $L_3$   is given by
\begin{equation} \label{id:angular_momentum_spherical}
\begin{aligned}
L_{3} &=  \phantom{-}\frac{1}{i} \frac{\partial \phantom{\phi}}{\partial \phi }.
\end{aligned}
\end{equation}


Recall that the self-adjoint operatos $L_\omega$ map, for any $k\in\N_0$, the space of spherical harmonics of degree $k$ into itself. We now introduce the specific basis of eigenfunctions of $L_3$ that will be useful for our purposes. 

To that aim, let us introduce a few preliminary facts on the associated Legendre polynomials.
Let $\ell,m \in \N_0$  with $0 \le m \le \ell$, and let $P_{\ell}^m$ be the associated Legendre polynomial
\begin{equation} \label{id:legendre_def}
P_\ell^m (t):=  (1-t^2)^{m/2} \,  \frac{d^{m} \phantom{x}}{dx^{m}}  P_\ell(t), \qquad t\in [-1,1],
\end{equation}
where $P_\ell$ is the Legendre polynomial of degree $\ell$, given by:\footnote{We follow  one of the standard conventions used in quantum mechanics for the associated Legendre polynomials and the spherical harmonics, see for example \cite[Chapter 12]{arfken}.}
\begin{equation} \label{id:demencial}
P_\ell(t) =  \frac{1}{2^\ell \ell!}  \frac{d^\ell \phantom{x}}{dx^\ell}  (t^2-1)^{\ell} =    2^{-\ell}   \sum_{s=0}^{\ell/2}     { (-1)^{s}}     \frac{( 2\ell-2s )!      }{ s! (\ell-s )! ( \ell-2s)!  }  t^{\ell-2s}, \qquad t\in [-1,1],
\end{equation}
see for example \cite[p. 44]{Lebedev}. For later use we remark that $P_\ell(1) = 1$. From  \eqref{id:legendre_def} it follows that
\begin{equation} \label{id:legendre_ell}
P_\ell^\ell(t) = c   (1-t^2)^{\ell/2},
\end{equation}
for an appropriate  constant $c>0$. Definition \eqref{id:legendre_def} can be extended to $-\ell \le m \le 0$  as follows:
\begin{equation*} 
P_\ell^m (t):= (-1)^{m} \frac{(\ell-m)!}{(\ell+m)!} P_\ell^{|m|} (t) \qquad t\in [-1,1].
\end{equation*}
Using the parametrization defined in \eqref{id:spherical_coord},  for $\ell \in\N_0$, $m\in \Z$ such that $-\ell \le m \le \ell$  the spherical harmonic $Y_{\ell,m}: \SS^2 \to \IC$ of degree $\ell$, order $m$  and north pole  $e_3$ is given by the formula
\begin{equation} \label{id:quantum_def_spherical_harmonic} 
Y_{\ell,m} (x(\theta,\phi)) = (-1)^m \sqrt{\frac{(2\ell+1)}{4\pi}}   \sqrt{\frac{(\ell-m)!}{(\ell+m)!}} P^{m}_{\ell}(\cos(\theta)) e^{im\phi}.
\end{equation}
With this convention the family $(Y_{\ell,m})_{\ell\in\N_0,-\ell\leq m \leq\ell}$ is an orthonormal basis of $L^2(\SS^{2})$.
Note that
\begin{equation} \label{id:sh_conj}
\ol{Y_{\ell,m}} (x)  = (-1)^m Y_{\ell,-m} (x)  \qquad -\ell \le m \le \ell,
\end{equation}
since the associated  Legendre polynomials are real.\footnote{The sign convention used here may seem cumbersome, but it is one of the usual conventions in quantum mechanics since it simplifies working with the ladder operators. When $m>0$, the $(-1)^m$ factor appearing in \eqref{id:quantum_def_spherical_harmonic} is known as  Condon–Shortley phase.} It turns out that the functions in this basis are eigenfunctions of $L_3$. In fact, it follows from   \eqref{id:angular_momentum_spherical} and \eqref{id:quantum_def_spherical_harmonic} that:
\begin{equation} \label{id:L3_eigenvalue}
 L_3 Y_{\ell,m} = m Y_{\ell,m}, \qquad \ell\in \N_0,\;-\ell\leq m \leq \ell.
\end{equation}
This shows that the spectrum of $L_3$ is $\Z$, and that the spectrum of $L_3|_{\H_\ell}$ is simple, since $\dim \H_\ell =2\ell +1$. 

The same holds, by virtue of \eqref{e:uequiam}, for any other angular momentum operator $L_\omega$ with $|\omega|=1$. For any given $P\in\SO(3)$ such that $P\omega=e_3$ the functions 
\begin{equation}\label{id:spherical_harmonic_omega}
Y^\omega_{\ell,m}:=R_{P^\T} Y_{\ell,m},\qquad \ell\in\N_0, \; -\ell\leq m \leq \ell,
\end{equation}
form an orthonormal basis of eigenfunctions of $L_\omega$. Different choices for the rotation $P$ give different basis of spherical harmonics. However, any two possible basis $(Y^\omega_{\ell,m})$ and $(\tilde{Y}^\omega_{\ell,m})$ obtained in this way are related by $\tilde{Y}^\omega_{\ell,m}=R_Q Y^\omega_{\ell,m}$ for some $Q\in\SO(3)_\omega$. Therefore, since 
\begin{equation}
    e^{-itL_\omega}\text{ is }2\pi\Z\text{-periodic},\quad \SO(3)_\omega=\{e^{-itL_\omega} \,:\, t\in [0,2\pi]\},
\end{equation}
one can always find $t_Q\in [0,2\pi]$ such that:
\begin{equation} \label{id:rot_behavior_omega}
\tilde{Y}^\omega_{\ell,m} =e^{-it_QL_\omega} Y_{\ell,m}^\omega = e^{-im t_Q} Y_{\ell,m}^\omega \quad \text{for all } \ell\in \N_0, \; -\ell \le m\le \ell. 
\end{equation}
This proves that  $\lambda_{k,\ell;\omega}[q]$ and $\m_{k,\ell;\omega}[q]$ defined respectively in \eqref{id:matrix_elements} and \eqref{id:moments} are actually independent of the choice of the particular rotation used to define $Y_{\ell,m}^\omega$.

\section{The Born approximation: the radial case} \label{sec:Born_radial}

In this section  we prove formula \eqref{id:fourier_from_moments} and  \Cref{main_thm:Born_aprox_radial}.


 We start by showing that the Fourier transform of a radial function $q\in L^1(\R^d)$ with compact support can be expressed in terms of the moments $\sigma_{k,1}[q]$ defined in \eqref{id:moments_radial}.
 If $q(x) = q_0(|x|)$ we have that
\[
\widehat{q}(\xi) = \int_{\R^n} q(x) e^{-ix\cdot\xi} \, dx =\frac{(2\pi)^{d/2}}{|\xi|^{(d-2)/2}} \int_{0}^\infty r^{(d-2)/2}J_{(d-2)/2}(r|\xi|) q_0(r) r \, dr
\]
(see \cite[Theorem 3.10]{steinweiss}), where  the Bessel function $J_\nu(t)$ is given by the series
\begin{equation} \label{id:Bessel}
J_\nu (t) = \sum_{k=0}^\infty \frac{(-1)^k}{k!\Gamma(k + \nu+1)} \left ( \frac{t}{2}\right )^{2k + \nu}  \qquad t\in\R_+,
\end{equation}
with $\R_+ = (0,\infty)$. 
  Using this  with $\nu= (d-2)/2$ in the formula of the Fourier transform, and moving the summation outside the integral |this can be justified provided $q$ has compact support| yields that 
\begin{equation*} 
 \widehat{q}(\xi)  =  \frac{2\pi^{d/2}}{|\SS^{d-1}|}  \sum_{k=0}^\infty \frac{(-1)^k}{k!\Gamma(k + d/2)} \left ( \frac{|\xi|}{2}\right )^{2k}  \int_{B} q(x)|x|^{2k} \, dx.
\end{equation*}
from which \eqref{id:fourier_from_moments} follows directly.
  As mentioned, the compact support of $q$ is essential, since one can easily find a Schwartz class function $f$ in $\R^d$ such that $\widehat{f}$ vanishes in a neighborhood of $\xi =0$, so cannot satisfy \eqref{id:fourier_from_moments}.

The identity \eqref{id:fourier_from_moments} provides a nice interpretation of the formula we obtain for the Born approximation \eqref{id:qexp_formula_int}:  from the point of view of reconstruction, $\widehat{\qe}(\xi)$ can bee seen as the function that one gets when substituting the unknown moments $\sigma_{k,1}[q]$  by the \textit{a priori} known quantity $\lambda_k[q] -k$ in \eqref{id:fourier_from_moments}.

\begin{proof}[Proof of \Cref{main_thm:Born_aprox_radial}] 
 Let $\zeta \in \mathcal V(d)$. One of the key facts that allows  to obtain the formula \eqref{id:qexp_formula_int} is  the following simple property. Using Taylor's expansion of the exponential one gets
\begin{equation} \label{eq:exp_taylor_int} 
e_{\zeta}(x) =  e^{\zeta \cdot x} =
\sum_{k=0}^\infty  \frac{1}{k!}   (\zeta \cdot x)^{k}  ,
\end{equation}  
 where we recall that $(\zeta \cdot x)^{k}$ is a spherical harmonic of degree $k$ (see \Cref{sec:spherical_harmonics}).  
For simplicity we write $\lambda_k = \lambda_k[q]$ in this proof.  Let $\zeta_1 ,\zeta_2 \in \mathcal V(d)$   such that
\begin{equation*} 
 \zeta_1 + \zeta_2 = -ih \xi+ r_h, 
 \end{equation*}
 with  $h>0$
Since $\zeta_1 \cdot \zeta_1 = \zeta_2 \cdot \zeta_2 =0$, we have that
\begin{equation} \label{id:sum_prod}
   \zeta_1 \cdot \zeta_2  = \frac{1}{2}(\zeta_1+\zeta_2)^2 =  -\frac{1}{2} |\xi|^2 h^2 + l_h,\quad l_h:=|r_h|(|r_h|+o(h))=o(h^2).
\end{equation}
On the other hand, as a consequence  of \eqref{eq:exp_taylor_int}  it follows that
\begin{equation} \label{id:taylor_exp_h}
e_{\zeta/h}(x) =  e^{\frac{\zeta}{h}\cdot x} =
\sum_{k=0}^\infty \frac{1}{k!}  \frac{1}{h^k}  (\zeta \cdot x)^{k}  ,
\end{equation}  
and since in the radial case $\Lambda_q$ is diagonal in the spherical harmonics, we have that
\begin{equation*} 
\Lambda_q(e_{\zeta/h})(x) =
\sum_{k=0}^\infty \frac{\lambda_k}{k!}  \frac{1}{h^k}  (\zeta \cdot x)^{k}  .
\end{equation*} 
Hence,    \Cref{lemma:prod_spherial_harmonics} and \eqref{id:sum_prod} yield
\begin{align*}
\left \langle e_{\zeta_1/h}, (\Lambda_q -\Lambda_0) e_{\zeta_2/h} \right \rangle  
   &=  \sum_{k=0}^\infty   \frac{\lambda_k  - k}{(h^k k!)^2}  \int_{\SS^{d-1}} (\zeta_1 \cdot x)^{k} (\zeta_2 \cdot x)^{k} \, dS(x) 
\\ &=  \sum_{k=0}^\infty    c_k \frac{\lambda_k  - k}{(h^k k!)^2} \frac{(\zeta_1 \cdot \zeta_2)^{k}}{2^k}   =  \sum_{k=0}^\infty  (-1)^k c_k   \frac{\lambda_k  - k}{( k!)^2} \left( \left(\frac{|\xi|}{2} \right)^2 - \frac{l_h}{h^2} \right)^{k}  .
\end{align*} 
Note that when $r_h=0$ one has $l_h=0$, the above quantity does not depend on $h$ and \eqref{id:q_exp_exact} follows. 
Otherwise, since $l_h =o(h^2)$, taking limits as $h\to 0^+$ and using \eqref{id:q_exp} and \eqref{id:c_k}  gives directly  \eqref{id:qexp_formula_int}. Combining \eqref{id:qexp_formula_int} with \eqref{id:fourier_from_moments} gives the second identity in the statement. This finishes the proof of the theorem.
\end{proof}


Let us mention a further interesting consequence of \Cref{main_thm:Born_aprox_radial}. The coefficients of the series \eqref{id:qexp_formula_int} satisfy for $k>1$ that
\[a_k := \frac{ (-1)^k}{k! \Gamma(k+d/2)}
  \left(\frac{|\xi|}{2}\right)^{2k}  = - \frac{|\xi|^2}{2k(2k + d-2)} a_{k-1}, \]
 which means, roughly, that $|a_k|$ increases  until $k\sim |\xi|$ and then decreases to 0 as $k\to \infty$. 
Hence, $|a_k|$ attains its (very large) maximum values when $k \sim |\xi|$. This implies that  the eigenvalues making a greater contribution to the value of $\widehat{\qe}(\xi)$ in  \eqref{id:qexp_formula_int} for a fixed $\xi$, are  the $\lambda_k[q]$ with $k\sim |\xi|$. Since $\lambda_k[q] -k$ is closer to $\sigma_{k,1}[q]$ as $k\to \infty$, this suggest that $\widehat{\qe}(\xi)$ will improve as an approximation for $\widehat{q}(\xi)$ as $|\xi| \to \infty$. Thus, the Born approximation should recover the leading discontinuities of the potential. This property is well known in backscattering and in  the fixed angle scattering problems, see among others  \cite{fix,back,PaivarintaSomersalo91,alberto_fix,AlbertoAna05}.  The recover of singularities  has also been studied in   the Calderón problem with $d=2$ in \cite{GLSSU2018,KLMS2007}.

\section{The Born approximation: the general case} \label{sec:main_Born_section_gen}
Here we consider the Born approximation in the general case. We will prove formula \eqref{id:qexpA_int} of \Cref{main_thm:qexpA}.  We restate this result in the following proposition.

 \begin{proposition}  \label{prop:qexpA}
Let $\xi \in \R^3\setminus\{0\}$ and $\omega := \xi/|\xi|$.    Assume  $q\in \mathcal{Q}_3$    and let $\eta_1,\eta_2 \in\SS^{2}$ be such that $\{\eta_1,\eta_2,\omega\}$ is a positively oriented orthonormal basis in $\R^3$. Assume $   \widehat{ \langle \qe \rangle}$ is given by \eqref{id:qexpA} with $\zeta_1$ and $\zeta_2$ as in \eqref{eq:zeta_h_nonradial}.   Then 
\begin{equation*}
  \widehat{ \langle \qe \rangle} (|\xi| \omega) =     
  \sum_{k=0}^\infty   \sum_{\ell=k}^\infty   (-i)^{\ell+k}  \mu_{k,\ell}  \, |\xi|^{\ell+k}  (\lambda_{k,\ell;\omega}[q] - k \delta_{k,\ell})    ,
\end{equation*}
where $ \lambda_{k,\ell;\omega}[q]$ is given by \eqref{id:matrix_elements}  and $\mu_{k,\ell}$ by \eqref{id:mu}.
\end{proposition}
The proof will follow from a series of intermediate results and is given at the end of this section.
We start with the following lemma.
\begin{lemma} Let  $\langle\Lambda_q\rangle_{\omega} $ given by \eqref{id:averaged_operator}. Then, for all $\omega \in\SS^2$ and $Q \in \SO(3)$ one has that
\begin{align} 
 \label{id:ADN_invariance_1}   &R_Q \langle\Lambda_q\rangle_{\omega} R_{Q^\T} 
 = \langle\Lambda_{R_Q(q)}\rangle_{Q(\omega)}, \\
\label{id:ADN_invariance_2}
   &[\langle\Lambda_q\rangle_{\omega},L_\omega]  =0.
\end{align}
In addition, if $\lambda_{k,\ell;\omega}[q]$ is given by \eqref{id:matrix_elements}, it holds  that
\begin{equation} \label{id:ADN_invariance_3}
    \left \langle  \ol{Y_{\ell,k}^\omega},\langle \Lambda_{ q}\rangle_{\omega}    Y_{k,k}^\omega \right \rangle = \left \langle  \ol{Y_{\ell,k}^\omega},  \Lambda_{ q}   Y_{k,k}^\omega \right \rangle  = \lambda_{k,\ell;\omega}[q].
\end{equation}
\end{lemma}
\begin{proof}
By \eqref{id:DN_invariance_int} and the first identity in \eqref{id:first_identity} it follows that
\[
    R_Q \langle\Lambda_q\rangle_{\omega} R_{Q^\T}  = \frac{1}{2\pi} \int_0^{2\pi} R_{Q} \Lambda_{e^{itL_\omega}(q)} R_{Q^\T} \, dt 
    =\frac{1}{2\pi} \int_0^{2\pi}   \Lambda_{R_Q\left( e^{itL_\omega}(q)\right)}  \, dt.
\]
 Now, since it holds that $R_Q L_\omega = L_{Q(\omega)} R_Q$ we also have that $R_Q e^{itL_\omega} = e^{itL_{Q(\omega)}} R_Q$. Thus
 \[
    R_Q \langle\Lambda_q\rangle_{\omega} R_{Q^\T}  
    =\frac{1}{2\pi} \int_0^{2\pi}   \Lambda_{ e^{itL_{Q(\omega)}}(R_Q(q))}  \, dt
    = \langle \Lambda_{R_Q(q)}\rangle_{Q(\omega)},
\]
where the last identity  follows again from   \eqref{id:first_identity}. This proves \eqref{id:ADN_invariance_1}. 

As we have seen in \Cref{sec:sphere_d}, the identity \eqref{id:ADN_invariance_2} holds if and only if    
$[\langle\Lambda_q\rangle_{\omega},R_Q]  =0$ for all $Q\in \SO(3)_\omega$. This in turn is equivalent to show that   $R_Q\langle\Lambda_q\rangle_{\omega}R_{Q^T} = \langle\Lambda_q\rangle_{\omega}$ for all $Q\in \SO(3)_\omega$. Now, for all $Q\in \SO(3)_\omega$, there is a $t_Q \in [0,2\pi]$ such that  $R_Q = e^{it_QL_{\omega}}$. Then, as in the proof of \eqref{id:ADN_invariance_1}, we have
\[
    R_Q \langle\Lambda_q\rangle_{\omega} R_{Q^\T}  
    =\frac{1}{2\pi} \int_0^{2\pi}   \Lambda_{R_Q\left( e^{itL_\omega}(q)\right)}  \, dt 
    =\frac{1}{2\pi} \int_0^{2\pi}   \Lambda_{ e^{i(t+t_Q)L_\omega}(q)}  \, dt = \langle\Lambda_q\rangle_{\omega},
\]
where the last identity follows from the $2\pi$-periodicity of the flow $e^{itL_\omega}$. This finishes the proof of \eqref{id:ADN_invariance_2}.

We now prove \eqref{id:ADN_invariance_3}. Using the definition  \eqref{id:averaged_operator} we have
\begin{multline*}
    \left \langle  \ol{Y_{\ell,k}^\omega},\langle \Lambda_{ q}\rangle_{\omega}    Y_{k,k}^\omega \right \rangle  
    = \frac{1}{2\pi} \int_0^{2\pi} \left \langle  e^{-it L_\omega } \ol{Y_{\ell,k}^\omega},  \Lambda_{ q} e^{-it L_\omega } Y_{k,k}^\omega \right \rangle  \, dt   \\
    =  \frac{1}{2\pi} \int_0^{2\pi}  e^{-it k } e^{+it k }  \left \langle   \ol{Y_{\ell,k}^\omega},  \Lambda_{ q}  Y_{k,k}^\omega \right \rangle  \, dt   
    =  \left \langle   \ol{Y_{\ell,k}^\omega},  \Lambda_{ q}  Y_{k,k}^\omega \right \rangle
\end{multline*}
using \eqref{id:sh_conj} and  \eqref{id:rot_behavior_omega}. This finishes the proof of the lemma.
\end{proof}
One of the advantages of the previous lemma is that it allows to reduce the arguments to the case of $\omega = e_3$, as we will   show later on.
The identity \eqref{id:ADN_invariance_2} states that $\langle\Lambda_{q}\rangle_{e_3}$ commutes with the angular momentum operator $L_3$  defined in \eqref{id:angular_momentum_cartesian}. For convenience we  fix  
\begin{equation} \label{eq:zeta_h_nonradial_e3}
\begin{alignedat}{1}
  \widetilde{\zeta}_1 &=  -e_1 - i\left( 1 -    h^2|\xi|^2 \right )^{1/2}e_2   -   ih |\xi|e_3,  \\
 \widetilde{\zeta}_2 &=  e_1 + ie_2 . 
 \end{alignedat}
\end{equation}
By   \eqref{id:taylor_exp_h}, to compute the right hand side of \eqref{id:qexpA} it is enough to know explicitly how $\langle\Lambda_{q}\rangle_{e_3}  -\Lambda_0$ acts on the spherical harmonics $Y^k_{\widetilde{\zeta}_2}(x) =(\widetilde{\zeta}_2 \cdot x)^k$ for all $k\in\N_0$. These functions can be expressed in terms of spherical harmonics with north pole $e_3$:
\begin{equation} \label{id:sphe_in_standard}
Y_{\widetilde{\zeta}_2}^\ell = (-1)^\ell c_\ell^{1/2} Y_{\ell,\ell}.
\end{equation}
To see this, check that $L_3 Y^\ell_{\widetilde{\zeta}_2} = \ell Y^\ell_{\widetilde{\zeta}_2}$, 
which implies that $Y_{\widetilde{\zeta}_2}^\ell=\alpha_\ell Y_{\ell,\ell}$ for some $\alpha_\ell\in\IC$. Since $|\widetilde{\zeta}_2| =\sqrt{2}$ and $Y_{\ell,\ell}$ is normalized,  necessarily $|\alpha_\ell|^2=c_\ell$ by \eqref{id:Yk_normalization}.  An explicit computation in the spherical coordinates given in \eqref{id:spherical_coord} to find the complex phase yields
\[
(\widetilde{\zeta}_2 \cdot x)^\ell = (x_1+i x_2)^\ell  = (\sin \theta)^{\ell}  e^{i\ell \phi}
\]
which by \eqref{id:legendre_ell}  and \eqref{id:quantum_def_spherical_harmonic}  implies \eqref{id:sphe_in_standard}.

As we will see, the right-hand side of \eqref{id:q_exp} can be computed for any self-adjoint operator $\A$ that commutes with $L_3$. Thus, here we will work with $\A$ in place of $ \langle\Lambda_{q}\rangle_{e_3} -\Lambda_0$. This   extra generality will come handy later on.

\begin{lemma} \label{lemma:A}
Let $\mathcal{A}$ be an operator on $L^2(\SS^{2})$ such that $D(\A)=H^1(\SS^2)$ and $[ \A,L_3] =0$.
Then, if $\widetilde{\zeta}_1,\widetilde{\zeta}_2$ are given by \eqref{eq:zeta_h_nonradial_e3}, we have
 \begin{equation*}
\lim_{h \to 0^+}\left \langle e_{\widetilde{\zeta}_1/h}, \A\,  e_{\widetilde{\zeta}_2/h} \right \rangle    =   
   \sum_{k=0}^\infty  \sum_{\ell=k}^\infty  (-i)^{\ell+k}  \mu_{k,\ell} |\xi|^{\ell+k} \,   \gamma_{k,\ell}    ,
\end{equation*}
where $\mu_{k,\ell}$ satisfies \eqref{id:mu} and
\begin{equation}   \label{id:matrix_elements_A}
 \gamma_{k,\ell}  = \left \langle \ol{Y_{\ell,k}}, \A  Y_{k,k} \right \rangle  \quad \text{for }  k,\ell \in \N_0,\; \ell\ge k. 
\end{equation}
\end{lemma}
\begin{proof} For convenience, in this proof we write $\zeta_i = \widetilde{\zeta}_i$, $i=1,2$.
Since $L_3$ and $\mathcal{A}$ commute, the eigenspaces  of $L_3$  are invariant subspaces of $\A$.  This implies that
\begin{equation} \label{id:vanishing_matrix_elements}
 \left \langle \ol{Y_{\ell',m'} },  \A   Y_{\ell,m}  \right \rangle = 0 \qquad \text{if } m\neq m'.
\end{equation}
By \eqref{id:Yk}, \eqref{id:sphe_in_standard} and \eqref{id:vanishing_matrix_elements} it follows   that
\begin{equation*}
 \A Y^k_{\zeta_2} =  (-1)^k c_k^{1/2} \sum_{\ell =k}^\infty  \gamma_{k, \ell} Y_{\ell,k}  \quad \text{for all } k\in \N_0.
\end{equation*}
    Using \eqref{id:Yk} and \eqref{id:taylor_exp_h}, we obtain that
\begin{equation} \label{id:Born_interm}
\lim_{h \to 0^+}\left \langle e_{\zeta_1/h}, \A\,  e_{\zeta_2/h} \right \rangle   =   
 \lim_{h\to 0^+}\sum_{k=0}^\infty (-1)^k c_k^{1/2} \sum_{\ell=k}^\infty    \frac{ \gamma_{k,\ell}}{ k!\ell!}  \frac{\left \langle  Y^\ell_{\zeta_1},  Y_{\ell,k}  \right \rangle}{h^{k+\ell}}  .
\end{equation} 
The key to compute the last limit is to understand the behavior of $ \left \langle Y^\ell_{\zeta_1},  Y_{\ell,k}  \right \rangle $ when $h \to 0^+$. 
\begin{lemma} \label{lemma:taylor_angular}
Assume that $\zeta_1   \in \mathcal V(3) $ is given by   \eqref{eq:zeta_h_nonradial_e3}. Then, for all $k,\ell \in \N_0$, $\ell \ge k $ we have that
\begin{equation*} 
\left \langle  Y^\ell_{\zeta_1},  Y_{\ell,k}  \right \rangle 
=  (-1)^\ell i^{\ell +k}  \,c_\ell^{1/2}    \sqrt{\frac{(2\ell)!}{ (\ell+k)!(\ell-k)!}}   \left(\frac{|\xi|}{2}\right )^{\ell+k}   h^{\ell+k} + O(h^{\ell+k+1})  ,
\end{equation*}
where $c_\ell^{1/2}$ was defined in \eqref{id:c_k}.
\end{lemma}
We postpone the proof of this lemma to proceed with the proof of \Cref{lemma:A}.
Combining  \Cref{lemma:taylor_angular} with  \eqref{id:Born_interm} we get
 \begin{align*} 
 \lim_{h \to 0^+}\left \langle e_{\zeta_1/h}, \A\,  e_{\zeta_2/h} \right \rangle      &=
    \sum_{k=0}^\infty \sum_{\ell=k}^\infty  (-i)^{\ell +k}  \,(c_\ell c_k)^{1/2}    \frac{ \gamma_{k,\ell}}{ k!\ell!}  \sqrt{\frac{(2\ell)!}{ (\ell+k)!(\ell-k)!}}   \left(\frac{|\xi|}{2}\right )^{\ell+k}    \\
    &=  \sum_{k=0}^\infty  \sum_{\ell=k}^\infty  (-i)^{\ell+k}  \mu_{k,\ell} |\xi|^{\ell+k} \,   \gamma_{k,\ell}    ,
\end{align*}
with
\[
\mu_{k,\ell} =     \frac{ \sqrt{c_\ell c_k}}{ k!\ell!}  \sqrt{\frac{(2\ell)!}{ (\ell+k)!(\ell-k)!}} 2^{-\ell -k}.
\]
To finish the proof we need to show that $\mu_{k,\ell}$ satisfies \eqref{id:mu}.
Combining   \eqref{id:c_k} and the identity
\begin{equation} \label{id:Gamma}
 2^{2k} k! \Gamma\left(k + 1/2 \right) = \sqrt{\pi} (2k)! ,
\end{equation}
 we get
\begin{align*}
\mu_{k,\ell} &=  2\pi^{3/2}       \sqrt{  \frac{1}{ 2^{ 2k} k!\Gamma(k+3/2)   } \frac{1}{   2^{2\ell} \ell! \Gamma(\ell+3/2) } \frac{(2\ell)!}{  (\ell+k)!(\ell-k)!}} \\
&= 2\pi^{3/2}         \sqrt{  \frac{1}{ \sqrt{\pi} (k+1/2) (2k)!   } \frac{1}{  \sqrt{\pi} (\ell+1/2) (2\ell)!  } \frac{(2\ell)!}{  (\ell+k)!(\ell-k)!}} ,
\end{align*}
which yields  \eqref{id:mu}. This finishes the proof the lemma.
\end{proof}

The proof of \Cref{lemma:taylor_angular} uses the so-called ladder operators $L_+$ and $L_-$:  
\begin{equation} \label{id:ladder}
L_+ = L_1 +iL_2, \qquad L_- = L_1-iL_2,
\end{equation}  
|see for example \cite[Section 8.2]{Teschl}. The term ladder operators is motivated by the fact that these operators act as right/left shift operators in the index $m$:
\begin{equation} \label{id:ladder_coeff}
\begin{aligned}
L_+ Y_{\ell,m} &= \sqrt{(\ell -m)(\ell+m+1)} Y_{\ell ,m+1}, \\
L_- Y_{\ell,m} &= \sqrt{(\ell+m)(\ell-m+1)}Y_{\ell,m-1}.
\end{aligned}
\end{equation} 
$L_1$ and $L_2$ have the following expressions in terms of ladder operators
\begin{equation} \label{id:L_ladder}
L_1 = \frac{1}{2}(L_+ +L_-)\quad \text{and} \quad  L_2 = \frac{1}{2i}(L_+ - L_-).
\end{equation}
\begin{proof}[Proof of \Cref{lemma:taylor_angular}]   
Let $\zeta = e_1 + ie_2$. We start by noticing that $\zeta_1 = -\O_{h|\xi|}(\zeta)$, where $\O_{h|\xi|}$ belongs to $\SO(3)$ and, in the $\{e_1,e_2,e_3\}$ basis, is given by the  matrix  
\[
\begin{pmatrix}
1 & 0 &   0 \\
0 & \cos{\epsilon} &  -\sin{\epsilon}\\
 0 & \sin{\epsilon} &  \cos{\epsilon}\\
\end{pmatrix}
\] 
where $\sin \epsilon = h|\xi|$ for $h$ small.
Therefore
\begin{align*}
Y^\ell_{\zeta_1} (x)
&=(\zeta_1 \cdot x)^\ell 
= (-\O_{h|\xi|}(\zeta)\cdot x)^\ell \\
&=(-1)^\ell(\zeta \cdot \O_{h|\xi|}^{\T} (x))^\ell 
=   c_\ell^{1/2}  R_{\O_{h|\xi|}} (Y_{\ell,\ell} )(x),
\end{align*}
 where to get the last equality we have used   \eqref{id:sphe_in_standard}. We must be careful now, since we recall that $\langle \centerdot , \centerdot\rangle$ is not the $L^2$ product |there is not complex conjugation, see \eqref{id:dual}. Since we want to use the $L^2(\SS^2)$ structure of certain integrals, it is convenient to use \eqref{id:sh_conj}   to write the quantity that we want to compute as follows
\begin{multline} \label{id:artificial_conj}
\left \langle Y^\ell_{\zeta_1},  Y_{\ell,k}  \right \rangle  
= \left \langle   Y_{\ell,k}  , Y^\ell_{\zeta_1}\right \rangle  
=  (-1)^{ k} \left \langle   \ol{Y_{\ell,-k} } , Y^\ell_{\zeta_1} \right \rangle \\
=  (-1)^{k } c_\ell^{1/2} \left \langle   \ol{Y_{\ell,-k} } , R_{\O_{h|\xi|}} (Y_{\ell,\ell} ) \right \rangle .
\end{multline} 
Observe that the $\O_{h|\xi|}$ matrices fix $e_1$, and   the generator of rotations that fix the $e_1$ vector is exactly the angular momentum operator  $L_1$ defined in \eqref{id:angular_momentum_cartesian}.
 This means that the operator $R_{\O_{h|\xi|}}$ can be generated from $L_1$ through the exponential map, as shown in \eqref{id:exponential_map}. Thus, the Taylor's expansion of the exponential
\[
R_{\O_{h|\xi|}}(f) = e^{-i\epsilon L_{1}} f =  \sum_{p=0}^\infty \frac{(-i)^{p}}{p!} \epsilon^p   L_1^p f,
\]
gives
\begin{equation}   \label{id:taylor_yk_epsilon}
\left \langle   \ol{Y_{\ell,-k} } , R_{\O_{h|\xi|}} (Y_{\ell,\ell} ) \right \rangle
=   \sum_{p=0}^{\ell+k} \frac{(-i)^{p}}{p!}  \epsilon^p \left \langle  \ol{Y_{\ell,-k} }, L_{1}^{p}( Y_{\ell,\ell} ) \right \rangle + O(\epsilon^{\ell+k+1}).
\end{equation}
The operator $L_{1}$ is not diagonal on the basis of spherical harmonics used here, so we recall the Ladder operators introduced in \eqref{id:ladder} and  \eqref{id:L_ladder}.  We have
\[  
L_{1} = \frac{1}{2}(L_+ + L_-) ,
\]
which makes computing $\left \langle  \ol{Y_{\ell,-k} }, L_{1}^{p}( Y_{\ell,\ell} ) \right \rangle $ straightforward using \eqref{id:ladder_coeff}. \\
On the one hand, in the case $0\le p<\ell+k$ for appropriate coefficients $a_j$ we have that
\[ L_{1}^{p} Y_{\ell,\ell}  = \sum_{j = 0}^{p} a_j Y_{\ell,\ell-j} ,
\]
so that
\[   
\left \langle  \ol{Y_{\ell,-k} }, L_{1}^{p}( Y_{\ell,\ell} ) \right \rangle 
= 0\qquad \text{for all }  0\le p < \ell +k.
\]
On the other hand, when $p= \ell +k$, notice that
\begin{align*}
(L_-)^{\ell+k} \,  Y_{\ell,\ell}  
&=  \left(\prod_{m = -k+1}^{\ell} \sqrt{(\ell+m)(\ell-m+1)} \right) Y_{\ell,-k}  
\\ &=    \sqrt{\frac{(\ell+k)!(2\ell)!}{(\ell-k)!}} \; Y_{\ell,-k} .
\end{align*}
This means that there is exactly one term in the expansion of $ L_{1}^{p} Y_{\ell,\ell} $ which does not vanish after taking the $L^2(\SS^{2}) $ product with $Y_{\ell,-k} $. Hence
one gets
\begin{equation*}   
\left \langle  \ol{Y_{\ell,-k} }, L_{1}^{\ell +k}( Y_{\ell,\ell} ) \right \rangle 
= \frac{1}{2^{\ell+k}}    \sqrt{\frac{(\ell+k)!(2\ell)!}{(\ell-k)!}}
\end{equation*}
Substituting this in \eqref{id:taylor_yk_epsilon}, and using that $ \sin \epsilon =  h|\xi|$, and hence $\epsilon = h|\xi|+ O(h^2)$, we obtain that
\begin{equation*}   
\left \langle   \ol{Y_{\ell,-k} } , R_{\O_{h|\xi|}} (Y_{\ell,\ell} ) \right \rangle
=   (-i)^{\ell +k} \sqrt{\frac{(2\ell)!}{ (\ell+k)!(\ell-k)!}} \left( \frac{|\xi|}{2} \right )^{\ell+k}   h^{\ell+k}   + O(h^{\ell+k+1}) .
\end{equation*}
Using this in \eqref{id:artificial_conj} finishes the proof of the lemma.
\end{proof}

We now prove the main result in this section.

\begin{proof}[Proof of \Cref{prop:qexpA}] 
Recall that $\zeta_1,\zeta_2$ and $\widetilde{\zeta}_1,\widetilde{\zeta}_2$ are given, respectively, by \eqref{eq:zeta_h_nonradial} and \eqref{eq:zeta_h_nonradial_e3}.
Since the basis $\{\eta_1,\eta_2,\omega\}$ is positively oriented, there is one $P\in \SO(3)$ such that 
 \[
 \{\eta_1,\eta_2,\omega\} = \{P(e_1),P(e_2),P(e_3)\},
 \]
 so that $ {\zeta_1} = P(\widetilde{\zeta}_1)$ and $ {\zeta_2} = P(\widetilde{\zeta}_2)$.
Since $
R_Q (e_\zeta) = e_{Q(\zeta)} $ we have
\begin{align*}
   \widehat{ \langle \qe \rangle} (\xi )  &= \lim_{h\to 0^+}  \left  \langle e_{  \zeta_1/h},  \left(\langle \Lambda_q \rangle_\omega-\Lambda_0\right ) ( e_{   \zeta_2/h}) \right  \rangle \\
  &= \lim_{h\to 0^+}  \left  \langle e_{ \widetilde{\zeta}_1/h},  R_{P^\T} \left(\langle \Lambda_q \rangle_\omega-\Lambda_0\right ) R_{P}( e_{ \widetilde{\zeta}_2/h}) \right  \rangle.
  \end{align*}
  
Thus, we take $\A = R_{P^\T} \left(\langle \Lambda_q \rangle_\omega-\Lambda_0\right )R_{P}$.  Since  $P^\T(\omega) = e_3$, by \eqref{id:ADN_invariance_1} we have that  
\[
\A =  \left(\langle \Lambda_{R_{P^\T}(q)} \rangle_{e_3}-\Lambda_0\right ) ,
\]
and hence $[\A,L_3] =0$ by \eqref{id:ADN_invariance_2}. 
 Applying \Cref{lemma:A} we get
\begin{equation*} 
 \widehat{\langle {\qe} \rangle}(\xi)
=  \lim_{h \to 0^+}   { \left  \langle e_{  \widetilde{\zeta}_1/h}, \A \, e_{ \widetilde{\zeta}_2/h} \right  \rangle}  
  = \sum_{k=0}^\infty  \sum_{\ell=k}^\infty  (-i)^{\ell+k}  \mu_{k,\ell} |\xi|^{\ell+k}  \gamma_{k,\ell}    ,
\end{equation*}
where now, by \eqref{id:spherical_harmonic_omega} and \eqref{id:matrix_elements}, 
\begin{equation*} 
\begin{aligned} 
 \gamma_{k,\ell}  
 = \left \langle \ol{Y_{\ell,k}},\A  Y_{k,k} \right \rangle 
&=  \left \langle R_P \ol{Y_{\ell,k}},(\langle \Lambda_{ q}\rangle_{\omega} -\Lambda_0)  R_P Y_{k,k} \right \rangle  \\
  &=  \left \langle  \ol{Y_{\ell,k}^\omega},\langle \Lambda_{ q}\rangle_{\omega}    Y_{k,k}^\omega \right \rangle -\delta_{k,\ell} k.
 \end{aligned}
\end{equation*}
since $P(e_3) = \omega$. By \eqref{id:ADN_invariance_3}, this finishes the proof of the theorem.
\end{proof}


\section{The structure of D-N map} \label{sec:spectrum}

This section is devoted to studying the spectrum and matrix elements of the Dirichlet to Neumann map. We will make extensive use of  the following version of Alessandrini's identity \cite{Alessandrini88}: for every $f,g\in H^{1/2}(\SS^{d-1})$ the following holds:
 \begin{equation} \label{id:alessandrini_0}
 \br{f,(\Lambda_{q}-\Lambda_{0})g} = \int_{B} q(x) u(x)v(x)\,dx, 
 \end{equation}
 where $u$ solves \eqref{id:calderon_q} and $v$ solves
\begin{equation*} 
\left\{
\begin{array}{llr}
\Delta v  &= 0  & \text{in }   B,\\
v |_{  \SS^{d-1}} &= f. &    \\
\end{array}\right.
\end{equation*}

\subsection{The radial case} \label{sec:radial_series} In this section we prove \Cref{main_thm:radial_series}. Let  $d \ge 2$ and consider  
a radial potential $q(x) = q_0(|x|)$ such that $q \mathcal \in \cQ_d$. If $g\in H^{1/2}(\partial\Omega)$ the direct problem  \eqref{id:calderon_q}
\begin{equation}\label{e:pot}
\left\{
\begin{array}{rcr}
-\Delta u(x)+q_0(|x|)u(x)=&0&\text{ for } x\in B,\\
u|_{  \partial B}=&g,& \\
\end{array}\right.
\end{equation} 
and the Dirichlet to Neumann map
\[
\Lambda_q(f)=\partial_\nu u|_{\partial B}
\]
is well-defined. By \eqref{id:DN_invariance_int} this map commutes with all rotations in $\SO(d)$ and therefore $\Lambda_q|_{\H_k}=\lambda_k[q]\Id_{\H_k}$ for every $k\in\N_0$. If $\varphi_k\in\H_k$ and   $g = \varphi_k$ in \eqref{e:pot},  then  $u(x) = b_k(|x|)\varphi_k(x/|x|)$ where $b_k$ solves
\begin{equation} \label{id:radial_b}
    -\frac{1}{r^{d-1}}\frac{d}{dr}(r^{d-1} \frac{d}{dr} b_k(r))+ \left(\frac{k(k+d-2)}{r^2}+q_0(r)\right)b_k(r)=0,
\end{equation}
 subject to the boundary conditions
\[
\lim_{r\rightarrow 0^+}b_k(r)<\infty,\quad b_k(1)=1.
\]
 Note that there exist a unique solution $b_k$ with these properties, otherwise there would be more than one function $u$ satisfying \eqref{e:pot}.
Thus,
\[
\Lambda_q(\varphi_k)=\frac{d \phantom{r} }{dr}b_k (1)\varphi_k,
\]
which shows that $\lambda_k[q] = \frac{d \phantom{r}}{dr} b_k(1) $. From now on we will write $\lambda_k = \lambda_k[q] $.

The solid spherical harmonic associated  to $\varphi_k \in \H_k$ is   the harmonic function  $\widetilde{\varphi}_k(x) =  |x|^k \varphi_k(x/|x|)$ with $x\in B$. Thus,  \eqref{id:alessandrini_0}  yields that
\begin{multline} \label{id:aless_eigenvalue}
\lambda_k  -k = \langle \ol{\varphi_k} , (\Lambda_q -\Lambda_0) \varphi_k \rangle \\
=  \int_{B} |x|^k \ol{\varphi_k} \left(\frac{x}{|x|} \right) q_0(|x|) b_k(|x|)\, \varphi_k \left(\frac{x}{|x|} \right) \, dx = \int_0^1  q_0(r)  b_k(r) r^{k}  r^{d-1} \, dr,
\end{multline}
provided that $\norm{\varphi_k}_{L^2(\SS^{d-1})} =1$.
The change of variables
\begin{equation*} 
v_k(t)=e^{-\frac{d-2}{2}t} b_k (e^{-t}),\quad b_k(r)=r^{-\frac{d-2}{2}}v_k(-\log r),
\end{equation*}
transforms equation \eqref{id:radial_b} into its Liouville normal form:
\begin{equation}\label{e:1dred}
-\frac{d^2 \phantom{t}}{dt^2}v_k(t)+V(t)v_k(t)=-\left(k+\frac{d-2}{2}\right)^2v_k(t),\quad t\in\R_+,
\end{equation}
subject to the boundary conditions:
\begin{equation*} 
v_k(0)=1, \quad v_k\in L^2(\R_+),
\end{equation*}
and with
\begin{equation}\label{e:1dpot}
V(-\log r)=r^2q_0(r),\quad V(t)=e^{-2t}q_0(e^{-t}).
\end{equation}
Note again, that there exist a unique solution $v_k$ of \eqref{e:1dred} with these properties. Otherwise, this would imply non-uniqueness for $b_k$. Also, if $\alpha = \max \esupp q_0$, by \eqref{e:1dpot}  $V$ is supported in $[-\log \alpha,\infty) \subseteq \R_+$, and we have the pointwise bound 
\begin{equation}\label{id:Vqbound}
 |V(-\log r)|  = r^2|q_0(r)| \le \alpha^2 \norm{q_0}_{L^\infty((0,1])},\quad \text{a.e. }r\in [0,1] .
\end{equation}
 
To simplify notation, let $\kappa_d \ge 0$ be given by 
\begin{equation} \label{id:kappa}
\kappa_d = k + \frac{d-2}{2}.
\end{equation}
It will be convenient to restate the eigenvalue equation \eqref{e:1dred} as an equation with homogeneous   boundary conditions. To do this, take $u_{\kappa_d}$ such that  
\begin{equation} \label{id:uk}
v_k(t) = u_{\kappa_d}(t) + e^{-{\kappa_d} t} .
\end{equation}
Then \eqref{e:1dred} becomes
\begin{equation} \label{e:1dred_2}
-\frac{d^2 \phantom{t}}{dt^2}u_{\kappa_d} (t) + {\kappa_d}^2 u_k(t) +V(t)u_{\kappa_d} (t)=- V(t) e^{-{\kappa_d} t},\quad t\in \R_+,
\end{equation}
with conditions
\begin{equation*} 
u_{\kappa_d}(0)=0, \quad u_{\kappa_d}\in L^2(\R_+).
\end{equation*}

On the other hand, using the same change of variables \eqref{e:1dpot} in \eqref{id:aless_eigenvalue} gives
\begin{equation*} 
\lambda_k -k =     \int_0^\infty   e^{- k t} V(t) v_k(t) e^{-\frac{d-2}{2}t} \, dt ,
\end{equation*}
which  by \eqref{id:kappa} and \eqref{id:uk} becomes
 \begin{equation} \label{id:aless_eigenvalue_2}
\lambda_k -k =    \int_0^\infty  e^{- 2{\kappa_d} t} V(t)    \, dt + \int_0^\infty   e^{- {\kappa_d} t} V(t) u_{\kappa_d}(t)   \, dt
\end{equation}

Let $\kappa \ge 0$. We  introduce the free resolvent   operator  $\cR(\kappa) : L^2(\R_+) \to L^2(\R_+)$ given by 
\[
\cR(\kappa)f =  \left (-\frac{d^2\phantom{t}}{dt^2}+ \kappa^2 \right )^{-1} f .
\]
This operator can be computed explicitly by the formula
\begin{equation} \label{id:res_1d_0}
\cR(\kappa)f(t) = \int_{0}^\infty  G(t,s;\kappa) f(s) \, ds,
\end{equation}
where, for   $t,s \in \R_+$,
\begin{equation} \label{id:res_1d}
 G(t,s;\kappa) = \frac{\sinh \left (\kappa \min\{t,s\} \right )}{\kappa} e^{-\kappa\max\{t,s\}}= \frac{1}{2\kappa}\left( e^{-\kappa|t-s|} - e^{-\kappa(t+s)}\right) ,
\end{equation}
  is the Green function of the half-line $\R_+$.
 
Write $e_{\kappa}(t):=e^{-\kappa t}$; then \eqref{e:1dred_2} is equivalent to the following integral equation
\[
u_{\kappa_d}(t) + \cR({\kappa_d})(V u_{\kappa_d})(t)  = -\cR({\kappa_d})\left (V e_{\kappa_d} \right )(t)  .
\]
Multiplying by $V$ one gets
\[
(I + V \cR({\kappa_d}) ) (V u_{\kappa_d}) = -V \cR ({\kappa_d}) \left (V e_{\kappa_d}  \right )   ,
\]
which can be solved with a Neumann series:
\begin{equation} \label{id:neumann_series}
V u_{\kappa_d} = \sum_{n=2}^\infty (-1)^{n-1}(V\cR(\kappa_d))^{n-1} \left (V e_{\kappa_d}  \right )
\end{equation}
provided  $\norm{V\cR(\kappa_d)}_{L^2(X) \to L^2(X)} <1$, where $X= [-\log \alpha,\infty)$.
Assuming this holds, \eqref{id:aless_eigenvalue_2} becomes
\begin{equation} \label{id:series_radial}
\lambda_k =  k + \int_0^\infty  e^{- 2{\kappa_d} t} V(t)    \, dt + \sum_{n=2}^\infty \sigma_{k,n},
\end{equation}
where for $n \ge 2$ the numbers $\sigma_{k,n}$ satisfy
\begin{equation} \label{id:lambda_term}
\sigma_{k,n} =  (-1)^{n-1}\int_0^\infty   e_{ \kappa_d}(t) (V \cR(\kappa_d))^{n-1} \left (Ve_{\kappa_d} \right )(t)\, dt.
\end{equation}

To prove   \Cref{main_thm:radial_series} first notice that undoing the previous change of variables and using \eqref{id:kappa}, we have that
\begin{equation} \label{id:first_term}
\int_0^\infty  e^{- 2{\kappa_d} t} V(t)    \, dt = \int_0^1 q_0(r) r^{2k+d-1} \, dr = \frac{1}{|\SS^{d-1}|}\int_B q(x) |x|^{2k} \, dx, 
\end{equation}
which is exactly the first order term in \eqref{id:eigenvalue_series}. 
Then the proof of \Cref{main_thm:radial_series} follows  from the next lemma.
\begin{lemma} \label{lemma:resolvent_estimate}
Let $V$ be given by \eqref{e:1dpot} and $X =  [-\log \alpha,\infty) $. Then
\begin{equation*}
\norm{V\cR(\kappa)}_{L^2(X) \to L^2(X)}  
\le \frac{\alpha^2}{\kappa^2}\norm{q_0}_{L^\infty((0,1 ])},
\end{equation*}
for all $\kappa>0$.
\end{lemma}
\begin{proof}
Since $e^{-\kappa|t-s|} \ge e^{-\kappa(t+s)}$ for all $t,s  \in \R_+$, by \eqref{id:res_1d} it holds that
\[
  | G(t,s;\kappa) | \le \frac{1}{2\kappa}e^{-\kappa|t-s|}.
\]
 Take $f \in L^2(X)$. Then by the previous estimate, \eqref{id:Vqbound} and \eqref{id:res_1d_0}, 
\[  
\left|V(t)\cR(\kappa)(f) (t)\right | 
\le  \alpha^2 \norm{q_0}_{L^\infty([0,1))}  \frac{1}{2\kappa} \int_0^\infty  e^{-\kappa|t-s|} |f(s)| \, ds.
\]
Thus, if we extend $f$  as $0$ outside its support we can extend the integrals to $\R$ and apply Minkowski inequality as follows
\begin{align*}
\norm{V\cR(\kappa)(f)}_{L^2(X)}^2 
&\le \frac{\alpha^4}{4\kappa^2}  \norm{q_0}_{L^\infty([0,1))}^2 \int_{-\infty}^\infty \left( \int_{-\infty}^\infty e^{-\kappa|t-s|} |f(s)| \, ds \right )^2 dt \\
&= \frac{\alpha^4}{4\kappa^2}  \norm{q_0}_{L^\infty([0,1))}^2 \int_{-\infty}^\infty \left( \int_{-\infty}^\infty e^{-\kappa|s|} |f(t-s)| \, ds \right )^2 dt \\
&\le \frac{\alpha^4}{4\kappa^2} \norm{q_0}_{L^\infty([0,1))}^2 \norm{f}_{L^2(X)}^2 \left( \int_{-\infty}^\infty e^{-\kappa|t|}  \, dt  \right)^2
\end{align*}
This finishes the proof since $\int_{-\infty}^\infty e^{-\kappa|t|}  \, dt   = \frac{2}{\kappa}$
\end{proof}

\begin{proof}[Proof of $\Cref{main_thm:radial_series}$]
As we will now show, the $\sigma_{k,n}[q]$ numbers that appear in  the statement are the $\sigma_{k,n}[q] = \sigma_{k,n} $ numbers given by \eqref{id:lambda_term} with $V (t)=e^{-2t}q_0(e^{-t})$ and $q(x) = q_0(|x|)$.
In fact, \Cref{lemma:resolvent_estimate} implies that $\norm{V\cR(\kappa)}_{L^2(X) \to L^2(X)}  <1$ when we have the condition $\kappa>\alpha \norm{q_0}_{L^\infty((0,1])}^{1/2}$. Thus \eqref{id:neumann_series} holds when  $k\ge 0$ and
\[
k > \alpha \norm{q_0}_{L^\infty((0,1])}^{1/2} - \frac{d-2}{2}.
\]
This together with \eqref{id:first_term} proves \eqref{id:eigenvalue_series}.

On the other hand, \eqref{id:lambda_term} and Cauchy-Schwarz inequality imply that
\begin{equation*} 
|\sigma_{k,n}| \le \norm{ e_{\kappa_d }}_{L^2(X)} \norm{V \cR(\kappa_d)}_{L^2(X) \to L^2(X)}^{n-1} \norm{ V e_{ \kappa_d }}_{L^2(X)}. 
\end{equation*}
Using \Cref{lemma:resolvent_estimate}, \eqref{id:Vqbound}  and that $\norm{ e_{\kappa_d }}_{L^2(X)} = (2\kappa_d)^{-1/2}\alpha^k$ gives
\begin{equation*} 
|\sigma_{k,n}| \le  \frac{\alpha^{2(\kappa_d +n)}}{2\kappa_d^{2n-1}}  \norm{q_0}_{L^\infty((0,1])}^n = \frac{\alpha^{2(k +n) +d-2}}{2 \left( k + \frac{d-2}{2}\right)^{2n-1}}  \norm{q_0}_{L^\infty((0,1])}^n.
\end{equation*}
This finishes the proof of the theorem.
\end{proof}

 We can now invert the change of variables to obtain an explicit formula for $\sigma_{k,n}[q]$ in terms of $q(x) = q_0(|x|)$. If $n\ge 2$ and $k+ (d-2)/2 >0$ we have that
\begin{multline} \label{id:series_terms}
\sigma_{k,n}[q] = (-1)^{n-1} \int_{0}^1 \dots \int_{0}^1 \, q_0(r_1) r_1^{k+d/2} q_0(r_2) r_2 \times  \dots \\ 
\dots \times q_0(r_{n-1}) r_{n-1} q_0(r_n) r_n^{k+d/2} F(r_1,r_2,k) \dots F(r_{n-1},r_{n},k)  \,  dr_1  \dots dr_n,
\end{multline}
where we have that, for $\beta,r,s>0$,
\[
F(r,s,\beta-(d-2)/2) = \frac{\min(r,s)^{\beta}}{\max(r,s)^{\beta}} - r^\beta s^\beta .
\]


\subsection{The general 3-dimensional case} \label{sec:non_radial_series}

In this section we will assume that $d=3$. Our goal is to understand the asymptotic behavior of the matrix elements of the operator $\Lambda_q-\Lambda_0$  that appear in the formula \eqref{id:qexpA_int}.

Let $u_{k,m}$ be the solution of
\begin{equation}\label{id:pot_2}
\left\{
\begin{array}{rlr}
-\Delta u_{k,m} + qu_{k,m} =& 0   &\text{ in }   B,\\
u_{k,m}|_{  \partial B} =& Y_{k,m}, &\phantom{,}     \\
\end{array}\right.
\end{equation} 
where $q \in \cQ_3$, and $Y_{k,m}$ is a spherical harmonic with north pole $e_3$, as defined in \eqref{id:quantum_def_spherical_harmonic}.  Equation \eqref{id:pot_2} has a unique solution $u_{k,m}$, and the Dirichlet to Neumann map $\Lambda_{q} Y_{k,m} = \partial_{\nu} u_{k,m}$
is well-defined.
The solid spherical harmonic associated  to $Y_{k,m}$ is the function 
\begin{equation} \label{id:solid_harmonic}
\widetilde{Y}_{k,m}(x) =  |x|^k Y_{k,m} (x/|x|) \qquad x\in B. 
\end{equation}  
Let $k,\ell \in \N_0$ and $m,n \in \N_0$ such that $-k \le m \le k$ and $-\ell \le n \le \ell$.  The Allessandrini identity  \eqref{id:alessandrini_0} yields that
\begin{equation} \label{id:aless_matrix_2}
 \langle \ol{Y_{\ell,n}} , (\Lambda_q -\Lambda_0)  Y_{k,m} \rangle \\
=  \int_{B} |x|^\ell \ol{Y_{\ell,n}}(x/|x|) q(x) u_{k,m}(x) \, dx .
\end{equation}
As in radial case, we introduce a function $v_{k,m}$ such that  
\begin{equation} \label{id:u_Y_v}
u_{k,m} = \widetilde{Y}_{k,m} + v_{k,m} .
\end{equation}
 Then    \eqref{id:pot_2} becomes
\begin{equation} \label{id:poisson_2}
\left\{
\begin{array}{rlr}
-\Delta v_{k,m}  + q v_{k,m}  =& -q\widetilde{Y}_{k,m}  &\text{ in }   B,\\
v_{k,m}|_{  \partial B} =& 0,\phantom{\widetilde{Y}_{k,m}}  &\phantom{,}     \\
\end{array}\right.
\end{equation} 
 We now introduce the resolvent operator $ \cR_0: L^2(B) \to H^1_0(B)$ given by  the solution operator of the Poisson problem in the ball. If $d \ge 3$ this operator is given by
\[ \cR_0(f)(x) = \int_{B} G(x,y) f(y) \, dy, \]
 where 
\[ 
G(x,y) =   \mathfrak{c}_d \left(\frac{1}{|x-y|}- \frac{|x|}{|x - y|x|^2 |} \right).
\]
Applying the resolvent $\cR_0$ in \eqref{id:poisson_2}, we obtain the integral equation
\begin{equation} \label{id:lip_sch}
v_{k,m}    =  \cR_0(q\widetilde{Y}_{k,m}) + \cR_0(q v_{k,m})  = \cR_0(qu_{k,m}). 
\end{equation}
Inserting this and \eqref{id:u_Y_v}   in \eqref{id:aless_matrix_2} gives
\begin{multline} \label{id:aless_matrix_3}
 \langle \ol{Y_{\ell,n}} , (\Lambda_q -\Lambda_0)  Y_{k,m} \rangle
=  \int_{B} |x|^{\ell + k} q(x)   \ol{Y_{\ell,n}}\left(\frac{x}{|x|}\right)  Y_{k,m}\left(\frac{x}{|x|}\right)   \, dx  \\ + \int_{B} |x|^{\ell  } \,  \ol{Y_{\ell,n}}\left(\frac{x}{|x|}\right) q(x) \cR_0(qu_{k,m}) (x) \, dx      .
\end{multline}
In general, if $\alpha:=\max_{x\in \esupp q} |x|$, the first term in \eqref{id:aless_matrix_3} satisfies the bound 
\begin{align}
\nonumber \left|\int_{B} |x|^{\ell + k} q(x)   \ol{Y_{\ell,n}}\left(\frac{x}{|x|}\right)  Y_{k,m}\left(\frac{x}{|x|}\right)   \, dx    \right| 
&\le   \int_0^\alpha r^{\ell + k +2} \left| \int_{\SS^2} q(r\theta)   \ol{Y_{\ell,n}}(\theta) Y_{k,m}(\theta)   \, dS(\theta) \right|   dr   \\
 \label{est:size_m} &\le   \norm{q} _{L^\infty(B)}  \int_0^\alpha r^{\ell + k +2} \, dr =  \norm{q} _{L^\infty(B)}  \frac{\alpha^{k+\ell +3}}{k+\ell+3}.
\end{align}
Notice that the resolvent operator is independent of $k$, in contrast with the radial case. This  prevents us from using the Neumann series approach  unless one imposes smallness assumptions on the potential.  Nonetheless, since we are mostly interested to know if the first term in the right hand side of \eqref{id:aless_matrix_3} is the dominant one as $k,\ell \to \infty$, we can still use the regularization  effect provided by the resolvent operator to obtain extra decay for the second term in  \eqref{id:aless_matrix_3}.

 \begin{proposition} \label{prop:first_order_nonradial}
Let  $q\in \cQ_3$, and let $\alpha:=\max_{x\in \esupp q} |x|$. 
Assume $k,\ell\in \N$ with $\ell \ge k$ and let $\lambda_{k,\ell;\omega}[q]$   and $\m_{k,\ell;\omega}[q] $, respectively, as in   \eqref{id:matrix_elements} and \eqref{id:moments}. Then, for all $\omega \in \SS^2$ we have that
\begin{equation} \label{est:m_decay}
|\m_{k,\ell;\omega}[q] | \lesssim 
\frac{\alpha^{k+\ell}}{k+\ell +1}\norm{q}_{L^\infty(\R^3)} ,
\end{equation}
and
\begin{multline}  \label{est:asymp_k}
\left| \lambda_{k,\ell;\omega}[q] - k \delta_{k,\ell} - \m_{k,\ell;\omega}[q] \right|   \\     
\lesssim   \frac{\alpha^{k+\ell}}{(\ell +1)\sqrt{k+1}}   \norm{ q}_{L^\infty (\R^3)}^2 \left (1 +  \norm{ q}_{L^\infty (\R^3)}\norm{\mathcal (-\Delta  + q)^{-1}}_{L^2(B) \to L^2(B)} \right ),   
\end{multline}
where $\delta_{k,\ell}$ stands for the Kronecker delta. 
The implicit constants are independent of $k,\ell$, $\alpha$ and $q$. 
\end{proposition}
As mentioned in the introduction, \eqref{est:m_decay} is sharp for $k=\ell$, as shown by taking the indicator function of the ball as $q$, so $\eqref{est:asymp_k}$ gives in an improvement in decay of order $k^{-1/2}$.
\begin{proof}
We will assume that $\omega = e_3$. Otherwise, one can take a rotation $P\in \SO(3)$ such that $P(e_3) = \omega$, and use that
\begin{equation} \label{id:rot_invariance_m_lambda}
\lambda_{k,\ell;\omega}[q] = \lambda_{k,\ell;e_3}[R_{P^\T}(q)] \quad \text{and} \quad \m_{k,\ell;\omega}[q] = \m_{k,\ell;e_3}[R_{P^\T}(q)],
\end{equation}
to reduce to the previous case.

First, notice that \eqref{est:m_decay} is the same as estimate \eqref{est:size_m} for $m=n=k$, so we have just to prove \eqref{est:asymp_k}.
By \eqref{id:matrix_elements}, we have that
\[
\left \langle \ol{Y_{\ell,k} }, (\Lambda_q-\Lambda_0) Y_{k,k}  \right \rangle  
  = \lambda_{\ell,k;\omega} - k \delta_{k,\ell}.
\] 
Therefore we need to estimate the second term in \eqref{id:aless_matrix_3} in the special case $m = n=k$:
\[
  \lambda_{k,\ell;\omega}[q] - k \delta_{k,\ell} - \m_{k,\ell;\omega}[q]   =  \int_{B_\alpha } |x|^{\ell  } \,  \ol{Y_{\ell,k} } \left(\frac{x}{|x|}\right) q(x) \cR_0(qu_{k,k}) (x) \, dx  =I,
\]
where  $u_{k,k}$ solves \eqref{id:pot_2} with  $m=k$ (recall $B_\alpha$ stands for of radius $\alpha$ centered at the origin where $q$ is supported). We denote the integral term on the right by $I$.
 Hence
\begin{align*}
|I|        &\le  \norm{ q \cR_0(qu_{k,k}))}_{L^\infty (B)} \int_{B_\alpha} |x|^\ell  \widetilde{Y}_{\ell,k} (x) \, dx 
       \\
&\le  \frac{\alpha^\ell}{\ell +1}  \norm{ q}_{L^\infty (\R^3)}    \norm{ \cR_0(qu_{k,k}))}_{L^\infty (B)} .
\end{align*}
We now use that in $\R^3$
\begin{equation} \label{e:resolvent_H-1}
 \norm{ \cR_0(f)}_{L^\infty (B)} \lesssim  \norm{\cR_0(f)}_{H^2(B)}  \lesssim  \norm{f}_{L^2(B)} \qquad f\in L^2(B).
\end{equation}
This yields
\begin{equation*}
|I|        \lesssim  \frac{\alpha^\ell}{\ell +1} \norm{ q}_{L^\infty (\R^3)}^2  \norm{  u_{k,k}}_{L^2 (B)}   .
\end{equation*}
To finish, we use that $u_{k,k} = v_{k,k} + \widetilde{Y}_{k,k}$ together with the identity   
\[
v_k = -\cR_q(q\widetilde{Y}_{k,k}),\qquad \cR_q:=(-\Delta  + q)^{-1},
\]
which follows from \eqref{id:poisson_2}. Note, that, since $0$ is not a Dirichlet eigenvalue  of $(-\Delta  + q)$ (recall \eqref{id:Qdef}), $\cR_q$ is well defined and is bounded in $L^2(B)$.
Using this  we get 
\begin{equation*}
|I|        \lesssim  \frac{\alpha^\ell}{\ell +1} \norm{ q}_{L^\infty (\R^3)}^2 \left (1 +  \norm{ q}_{L^\infty (\R^3)}\norm{\mathcal R_q}_{L^2(B) \to L^2(B)} \right )  \norm{\widetilde{Y}_{k,k} }_{L^2(B_\alpha)}    .
\end{equation*}
 Since  for all $k \in \N_0$ and $m \in \Z$ with $|m|\le k$, direct integration of the solid harmonic yields
\begin{equation*}  
\norm{\widetilde{Y}_{k,m} }_{L^2(B_\alpha)}^2  = \frac{\alpha^{2k}}{2k+2},
\end{equation*}
one concludes that
\begin{equation*} 
\left| I \right|       
\lesssim   \frac{\alpha^{k+\ell}}{(\ell +1)\sqrt{k+1}}   \norm{ q}_{L^\infty (\R^3)}^2 \left (1 +  \norm{ q}_{L^\infty (\R^3)}\norm{\mathcal R_q}_{L^2(B) \to L^2(B)} \right ) ,  
\end{equation*} 
which proves \eqref{est:asymp_k}.
\end{proof}


\section{Proof of the Fourier transform formula}  \label{sec:fourierTF}

On this section we prove representation \eqref{id:Fourier_representation_int}  for the Fourier transform.  This finishes the proof of \Cref{main_thm:qexpA}, since \eqref{id:qexpA_int} has been proved by \Cref{prop:qexpA} and \eqref{est:asymp_k_int}  by \Cref{prop:first_order_nonradial}. Let us state again 
the result given by \eqref{id:Fourier_representation_int}.
\begin{proposition}  \label{prop:general_fourier}
Take $\xi \in \R^3\setminus\{0\}$ and $\omega = \xi/|\xi|$. Assume that $q\in L^1(\R^3)$ has compact support. Then  
\begin{equation} \label{id:Fourier_representation}
\widehat{q} (\xi)   =   
  \sum_{k=0}^\infty  \sum_{\ell=k}^\infty     (-i)^{\ell+k} \frac{4\pi  }{\sqrt{(2k+1)(2\ell+1)}} \frac{  1}{ \sqrt{ (2k)!(\ell-k)! (k+\ell)!}}  |\xi|^{\ell+k} \,   \m_{k,\ell;\omega}[q] ,
\end{equation}
where the series converges absolutely.
\end{proposition}
Recall that   $\m_{k,\ell;\omega}[q]$ was defined in \eqref{id:moments}.
\begin{proof}[Proof of   \Cref{main_thm:qexpA}]
It follows directly putting together  \Cref{prop:qexpA,prop:first_order_nonradial,prop:general_fourier}.
\end{proof}

 The key to  prove \Cref{prop:general_fourier}   is to express products of spherical harmonics as linear combinations of terms with one spherical harmonic. This is provided by the following lemma.
\begin{lemma} \label{lemma:sphe_harm_products}
Let $k,\ell\in \N_0$ and consider the set $A(k,\ell) \subseteq \N_0$ of numbers $n$ with $|\ell-k|\le n \le \ell+k$ such that $\ell+k + n = 2g$ for some $g \in \N_0$. Then it holds that
\begin{equation*}
\ol{Y_{\ell,k}}(x) Y_{k,k}(x)  = \sum_{n \in A(k,\ell)} (-1)^{g +k -n}   \frac{ (2n+1)}{4\pi} G(k,\ell,n) P_{n}(x\cdot e_3) 
\end{equation*}
where $P_n$ stands for the Legendre polynomial of degree $n$, 
\begin{equation*}
G(k,\ell,n)=  F(g) \frac{(-k +\ell +n)!}{(k +\ell +n +1)!}    \sqrt{ (2k+1)(2\ell+1)}  \sqrt{ \frac{(k + \ell)!(2k)! }{  (\ell -k )!}},
\end{equation*}
and
\begin{equation} \label{id:fg}
F(g) = \frac{g!}{(g-k )! (g-\ell )! (g-n )!} .
\end{equation}
\end{lemma}
The proof of this lemma involves computations with the Clebsch-Gordan coefficients, so we leave the details  for the \cref{sec:appendix_cl_coeff}.

\begin{proof}[Proof of $\Cref{prop:general_fourier}$]
  The absolute convergence follows easily by  the estimate \eqref{est:m_decay}. The $\mu_{k,\ell;\omega}[q]$ grow exponentially with $k$ |here we don't assume $q$ supported in $B$, so $\alpha$ can be larger that $1$| but this is  compensated by the factors $(2k)!(\ell-k)! (k+\ell)!$ in the denominator.

Let $x = |x|\theta $ and $\xi = |\xi| \omega$, where $\omega,\theta \in \SS^{2}$.  By \eqref{id:rot_invariance_m_lambda} and the invariance properties of the Fourier tranform under rotations, it is enough to prove the case $\omega= e_3$. 
We recall that
\begin{equation*} 
\m_{k,\ell;e_3}[q] = \int_{B} |x|^{\ell + k} q(x)  \ol{Y_{\ell,k} } \left(\frac{x}{|x|}\right)   Y_{k,k}  \left(\frac{x}{|x|}\right)   \, dx  .
\end{equation*}
In the right hand side  of \eqref{id:Fourier_representation}, let us commute the integral of $q$ that appears in the previous expression for  $\m_{k,\ell;e_3}$   with the summation in $k$ and $\ell$. 
  This is justified by the absolute convergence of the series and the absolute value estimate \eqref{est:m_decay} for the integrals that define the moments.
Thus, the identity \eqref{id:Fourier_representation} can be written as 
\[ 
\widehat{q} (\xi)  = \int_{B}q(x) a(x,\xi) \, dx,
\]
where 
\begin{multline} \label{id:a_def}
a(x,\xi)   = \sum_{k=0}^\infty  \sum_{\ell=k}^\infty   (-i)^{\ell+k} \times \\
\times \frac{4\pi  }{\sqrt{(2k+1)(2\ell+1)}} \frac{  1}{ \sqrt{ (2k)!(\ell-k)! (k+\ell)!}} |\xi|^{\ell+k}  |x|^{\ell +k } \, \ol{Y_{\ell,k} } \left(\frac{x}{|x|}\right)   Y_{k,k}  \left(\frac{x}{|x|}\right).
\end{multline}
 As a consequence,  to prove  that \eqref{id:Fourier_representation} holds true with $\xi = |\xi|e_3$, it is enough to show that
\begin{equation*}
a(x,\xi) =  e^{ -i\xi \cdot x} = e^{-i |\xi||x| e_3 \cdot \theta} .
\end{equation*}
On the other hand, the expansion of a plane wave   in terms of spherical waves is given by the classical formula
\[
e^{ -i\xi \cdot x} = e^{-i |\xi||x| e_3 \cdot \theta} = \sum_{n=0}^\infty  (-i)^n (2n+1)  P_{n}(\theta\cdot e_3)  j_n(|\xi| |x|),
\]
where $j_n$ is the Bessel spherical function given by
\[
j_n(t)  = \sqrt{\frac{\pi}{2x}} J_{n +\frac{1}{2}}(t) = \frac{\sqrt{\pi}}{2} \sum_{m=0}^\infty \frac{(-1)^m}{m!\Gamma(m + n+3/2)} \left ( \frac{t}{2}\right )^{2m + n} ,
\]
using \eqref{id:Bessel}.  Thus
\begin{equation*}
e^{ -i\xi \cdot x} =   \frac{\sqrt{\pi}}{2} \sum_{n=0}^\infty  (-i)^n (2n+1)  P_{n}(\theta\cdot e_3) \sum_{m=0}^\infty \frac{(-1)^m}{m!\Gamma(m + n+3/2)} \left ( \frac{|\xi||x|}{2}\right )^{2m + n}  .
\end{equation*}
Therefore,    \eqref{id:Fourier_representation} will be proved for $\xi= |\xi|e_3$ if we show that
\begin{equation} \label{id:mega_identity}
a(x,\xi)
=  \frac{\sqrt{\pi}}{2} \sum_{n=0}^\infty  (-i)^n (2n+1)  P_{n}(\theta\cdot e_3) \sum_{m=0}^\infty \frac{(-1)^m}{m!\Gamma(m + n+3/2)} \left ( \frac{|\xi||x|}{2}\right )^{2m + n}.
\end{equation}
Notice that this does not follow from the addition theorem of spherical harmonics since the right hand side of \eqref{id:a_def} involves spherical harmonics of different degrees.


   \Cref{lemma:sphe_harm_products} implies that
 \begin{multline*}
   \frac{4\pi \, 2^{k +\ell}}{\sqrt{(2k+1)(2\ell+1)}} \frac{  1}{ \sqrt{ (2k)!(\ell-k)! (k+\ell)!}} 
 \ol{Y_{\ell,k}}(\theta) Y_{k,k}(\theta)  \\
 = \sum_{n \in A(k,\ell)} (-1)^{g+k-n} (2n+1) P_{n}(\theta\cdot e_3)  F(g) \frac{(-k +\ell +n)!}{(k +\ell +n +1)!}  \frac{  \, 2^{k +\ell}}{  (\ell-k)!}  ,
\end{multline*}
for all $\theta \in \SS^2$. Going back to   \eqref{id:a_def},  we have proved that
\begin{equation*}  
 a(x,\xi) = 
 \sum_{k=0}^\infty  \sum_{\ell=k}^\infty  \sum_{n \in A(k,\ell)} b(k,\ell,n,),
\end{equation*}
where
\begin{multline} \label{id:b}
b(k,\ell,n) =   (2n+1) P_{n}(\theta\cdot e_3)  i^{k+\ell}    \left(\frac{|\xi||x|}{2} \right)^{k+\ell}   \times\\
\times { (-1)^{g+\ell-n}}  F(g)   \frac{(-k +\ell +n)!}{(k +\ell +n +1)!}  \frac{   2^{k +\ell}}{  (\ell-k)!}.
\end{multline}
We now commute the summation in $\ell$ and in $n$.  To do this, first we change the summation in $\ell$ for a summation in a new index $p= \ell +k$. Recall that $n\in A(k,\ell)$ if $\ell-k \le   n \le \ell +k$ and that $\ell +k + n$ is even. This now becomes the condition  $n\in B(k,p) \subset \N_0$, where $B(k,p) $ is the set of    numbers such that $p-2k \le n \le p$ and   $p+n $ is even. Then
\[
a(x,\xi) 
 = \sum_{k=0}^\infty \,  \sum_{p=2k}^\infty \, \sum_{n \in B(k,p)} b(k,p-k,n)
 = \sum_{k=0}^\infty \,  \sum_{n=0}^\infty  \, \sum_{p \in D(n,k)}    b(k,p-k,n),
\]
where $p \in D(n,k)\subset \N_0$ if $n\le p \le 2k +n$, $p\ge 2k$,  and $p+n$ is even. 
Observe now that  the factor   $|x|^{\ell+k} = |x|^{p}$ appears in all the previous expressions, while in \eqref{id:mega_identity} we have the powers  $|x|^{2m+n}$. This motivates the introduction another change in the summation parameters by putting $p=2m+n$. Notice that this is natural: we have that $p+n$ is even, and hence so must be $2m = p-n$. Also,  $p\ge n$, so it holds that $m\ge 0$.  This yields
\[
a(x,\xi) 
 = \sum_{k=0}^\infty  \sum_{n=0}^\infty \sum_{\substack{0\le m \le k\\ m \ge k-n/2 }}     b(k,2m-k+n,n) 
 =     \sum_{n=0}^\infty \sum_{m=0}^\infty \sum_{k =m}^{m+n/2} b(k,2m-k+n,n).
\]
We now substitute $\ell = 2m +n-k$ in the previous formulas. Since $2g = k+\ell+n$ we get   that $g= m+n$. Using this in \eqref{id:fg} and in \eqref{id:b} gives
\begin{multline*}  
 a(x,\xi) = 
 \sum_{n=0}^\infty  P_{n}(\theta\cdot e_3) \sum_{m=0}^\infty  { (-1)^{n}} \, i^{2m+n}(2n+1) \sum_{k =m}^{m+n/2}  { (-1)^{k+m}}   \left(\frac{|\xi||x|}{2} \right)^{2m+n}   \times \\ 
 \times     \frac{(2m +2n-2k  )!}{(2m+2n +1)!}  \frac{    2^{2m+n}}{  (2m+n-2k)!} \frac{(m+n)!}{(m+n-k )! (k-m)! m!}.
\end{multline*} 
By \eqref{id:Gamma} we have that
\begin{align*}
(2m+2n +1)! &= (2m+2n+1)(2m+2n)! \\
&= \frac{2}{\sqrt{\pi}} 2^{2m+2n} (m+n) ! \left(m+n+1/2\right) \Gamma \left(n+m+1/2\right ) \\
 &= \frac{2}{\sqrt{\pi}}  2^{2m+2n} (m+n) !  \Gamma \left(n + m + 3/2\right ) .
\end{align*}
Thus
\begin{multline*}  
 a(x,\xi) = 
 \frac{\sqrt{\pi}}{2} \sum_{n=0}^\infty  (-i)^n(2n+1) P_{n}(\theta\cdot e_3)  \sum_{m=0}^\infty \frac{(-1)^m}{m!\Gamma(m + n+3/2)} \left(\frac{|\xi||x|}{2} \right)^{2m+n} \times \\ 
 \times  2^{-n}  \sum_{k =m}^{m+n/2}  { (-1)^{k+m}}   \frac{(2m +2n-2k  )!      }{  (2m+n-2k)!  (m+n-k )! (k-m)!}  .
\end{multline*}
Putting $k =s+m$ with $s\in \N_0$, yields
\begin{multline*}  
 a(x,\xi) = 
 \frac{\sqrt{\pi}}{2} \sum_{n=0}^\infty (-i)^n(2n+1) P_{n}(\theta\cdot e_3)   \sum_{m=0}^\infty \frac{(-1)^m}{m!\Gamma(m + n+3/2)} \left(\frac{|\xi||x|}{2} \right)^{2m+n} \times \\ 
 \times  2^{-n}   \sum_{s=0}^{n/2}     { (-1)^{s}}     \frac{( 2n-2s )!      }{ s! (n-s )! ( n-2s)!  }  .
\end{multline*}
 Legendre polynomials satisfy $P_n (1) = 1$, so \eqref{id:demencial} implies that
\[
 2^{-n}   \sum_{s=0}^{n/2}     { (-1)^{s}}     \frac{( 2n-2s )!      }{ s! (n-s )! ( n-2s)!  } =  P_n (1) = 1 \quad \text{for all }  n\in \N.
\]
This proves \eqref{id:mega_identity} and hence finishes the proof of the proposition.
\end{proof}

\appendix

\section{Products of spherical harmonics} \label{sec:appendix_cl_coeff}

In this section we prove \Cref{lemma:sphe_harm_products}. To do this we use the following formula
\begin{multline} \label{id:rose_formula}
Y_{\ell_1,m_1}(x) Y_{\ell_2,m_2}(x) = 
\\  \sum_{\ell = |\ell_1-\ell_2| }^{\ell_1+\ell_2} \left( \frac{(2\ell_1+1)(2\ell_2+1)}{4\pi (2\ell+1)}\right)^{1/2}  C(\ell_1 \, \ell_2\, \ell; \, m_1\, m_2)C(\ell_1\, \ell_2\, \ell; \,0\,0) Y_{\ell,m_1+m_2}(x)
\end{multline}
|see for example \cite[formula (4.32)]{rose}| where the $C(\ell_1 \,\ell_2 \,\ell; \,m_1 \,m_2)$ are the Clebsch-Gordan coefficients. 
These coefficients can be computed in terms of $\ell_1 ,\ell_2, \ell, m_1$, and $m_2$ thorough an explicit but quite convoluted expression known as Racah formula.  
Fortunately, the special cases 
 $m_2 = -\ell_2$ and $m_1 =m_2 =0$   are much simpler than the general case. In fact, from formula  (3.29)   of \cite{rose} we obtain
\begin{equation} \label{id:c-g_coeff_1}
C(\ell\, k\, n;  \, k\, -k)=  \sqrt{2n+1}
 \left[  \frac{(2k)!(-k +\ell +n)!}{(k +\ell +n +1)!(k -\ell +n )!} \frac{(k + \ell)!}{(k +\ell -n )! (\ell -k )!} \right]^{\frac{1}{2}}.
\end{equation}
On the other hand, for values of $n\in \N_0$ such that $\ell + k +n$ is even,  formula  (3.32) of \cite{rose}  gives
\begin{multline} \label{id:c-g_coeff_2}
C(\ell\, k\, n;  \, 0\, 0) =  
\\ (-1)^{\frac{1}{2}(k+\ell-n)} \sqrt{2n+1} \left[  \frac{(-k +\ell +n)!(k -\ell +n)!(k +\ell -n)!}{(k +\ell +n +1)!}  \right]^{\frac{1}{2}} F(g),
\end{multline}
where $ 2g = \ell + k +n$ and $F(g)$ is given by \eqref{id:fg}. If $\ell + k +n$ is odd, it holds that $C(\ell\, k\, n;  \, 0\, 0)=0$. We can now prove the lemma.

\begin{proof}[Proof of $\Cref{lemma:sphe_harm_products}$]
First, by \eqref{id:spherical_harmonic_omega} and \eqref{id:rose_formula} one immediately has that
\begin{multline*}
Y_{\ell_1,m_1} (x) Y_{\ell_2,m_2} (x) = 
\\  \sum_{\ell = |\ell_1-\ell_2| }^{\ell_1+\ell_2} \left( \frac{(2\ell_1+1)(2\ell_2+1)}{4\pi (2\ell+1)}\right)^{1/2} C(\ell_1 \, \ell_2\, \ell; \, m_1\, m_2)C(\ell_1\, \ell_2\, \ell; \,0\,0)  Y_{\ell,m_1+m_2} (x).
\end{multline*}
On the other hand, by \eqref{id:quantum_def_spherical_harmonic}, the quantity $\ol{Y_{\ell,k}}(x) Y_{k,k}  (x) $ is real for all $x\in\SS^2$. Hence, by \eqref{id:sh_conj}   we have that
\[
\ol{Y_{\ell,k}} (x) Y_{k,k}   (x)  = Y_{\ell,k} (x) \ol{Y_{k,k}}   (x) = (-1)^k Y_{\ell,k} (x) Y_{k,-k} (x) \quad x\in \SS^2.
\]
Therefore in our case we get
\begin{equation*}
\ol{Y_{\ell,k}} (x) Y_{k,k}   (x)  = 
\\  \sum_{n \in A(k,\ell) } (-1)^k \frac{\sqrt{(2\ell+1)(2k+1)}}{4\pi  }  C(\ell \, k\, n; \, k\, -k)C(\ell\, k\, n; \,0\,0)   P_n (x \cdot e_3),
\end{equation*}
since $Y_{n,0}(x) = \sqrt{\frac{2n+1}{4\pi}}P_n(x\cdot e_3)$  by \eqref{id:quantum_def_spherical_harmonic}. The set $A(k,\ell)$, defined in the statement of the lemma, appears when using that   $C(\ell\, k\, n;  \, 0\, 0)=0$ if $\ell + k +n$ is odd, as stated previously.
Therefore, inserting \eqref{id:c-g_coeff_1} and \eqref{id:c-g_coeff_2} in the previous identity proves the lemma.
\end{proof}

\bibliographystyle{myalpha}

\bibliography{references_dnmap}

\end{document}